\documentclass[11pt]{amsart}
\usepackage{amsfonts,latexsym,amsthm,amssymb,amsmath,amscd,euscript,tikz, tikz-cd}
\usepackage[alphabetic, msc-links, bibtex-style, nobysame]{amsrefs}
 \usepackage{amsmath} 
 \usepackage{mathtools,slashed}
\usepackage{stackengine}
\usepackage{framed}
\usepackage{xfrac}
\usepackage[makeroom]{cancel}
\usepackage{faktor}
\usepackage{braket}
\usepackage{pgf,tikz,pgfplots}
\usepgfplotslibrary{groupplots}
\pgfplotsset{compat=1.18}
\usepgfplotslibrary{fillbetween} 
\usepackage{mathrsfs}
\usetikzlibrary{arrows}
\definecolor{wrwrwr}{rgb}{0.3803921568627451,0.3803921568627451,0.3803921568627451}
\definecolor{rvwvcq}{rgb}{0.08235294117647059,0.396078431372549,0.7529411764705882}
\definecolor{mblue}{rgb}{0.2, 0.3, 0.8}
\definecolor{morange}{rgb}{1, 0.5, 0}
\definecolor{mgreen}{rgb}{0.1, 0.4, 0.2}
\definecolor{mred}{rgb}{0.5, 0, 0}
\definecolor{ForestGreen}{RGB}{34,139,34}
\usepackage{float}

\numberwithin{equation}{section}

\usepackage{lmodern}
\usepackage{alphabeta}
\usepackage{supertabular}
\usepackage{amssymb}
\usepackage{enumerate}
\usepackage{stmaryrd}
\usepackage{bbm}
\usepackage{mathtools}
\usepackage{setspace}
\usepackage{tikz,tikz-cd}
\usetikzlibrary{matrix,calc,positioning,arrows,decorations.pathreplacing,patterns,knots}

\newtheorem{theorem}{{Theorem}}[section]
\newtheorem*{theorem*}{Theorem}
\newtheorem{lemma}[theorem]{Lemma}
\newtheorem{proposition}[theorem]{Proposition}

\newtheorem{example}[theorem]{Example}
\newtheorem{corollary}[theorem]{Corollary}
\newtheorem*{corollary*}{Corollary}

\theoremstyle{definition}
\newtheorem{definition}{Definition}
\newtheorem{remark}{Remark}

\usepackage{lmodern,url,enumerate,mathtools}
\usepackage[hmargin = 1in,vmargin=1in]{geometry}
\usepackage{graphicx}
\usepackage{subcaption}

\usepackage{hyperref}
    \hypersetup{colorlinks=true,citecolor=ForestGreen,linkcolor = blue,urlcolor =black,linkbordercolor={1 0 0}}

\newcommand{\mr}[1]{{\rm #1}}
\newcommand{\nc}{\newcommand}
\nc{\sn}{\mr{sn}}
\nc{\cn}{\mr{cn}}
\nc{\dn}{\mr{dn}}

\nc{\on}{\operatorname}
\nc{\p}{\partial}
\nc{\ol}{\overline}
\nc{\ul}{\underline}
\nc{\pa}{\partial}

\nc{\pb}{\partial_b}
\nc{\pc}{\partial_c}
\nc{\pd}{\partial_d}
\nc{\pe}{\partial_e}
\nc{\pf}{\partial_f}
\nc{\pg}{\partial_g}
\nc{\ph}{\partial_h}
\nc{\pari}{\partial_i}
\nc{\pj}{\partial_j}
\nc{\pk}{\partial_k}
\nc{\pl}{\partial_l}
\nc{\pell}{\partial_\ell}
\nc{\parm}{\partial_m}
\nc{\pn}{\partial_n}
\nc{\po}{\partial_o}
\nc{\pp}{\partial_p}
\nc{\pq}{\partial_q}
\nc{\pr}{\partial_r}
\nc{\ps}{\partial_s}
\nc{\pt}{\partial_t}
\nc{\pu}{\partial_u}
\nc{\pv}{\partial_v}
\nc{\pw}{\partial_w}
\nc{\px}{\partial_x}
\nc{\py}{\partial_y}
\nc{\pz}{\partial_z}

\numberwithin{equation}{section}
\makeatletter
\@addtoreset{equation}{section}
\makeatother

\title{Spacetime Positive Mass Theorem with boundary for Multiple Time Dimensions} %

\author{Changwen Fang}
\address{Department of Mathematics, Columbia University in the City of New York}
\email{cf2971@columbia.edu}

\begin{document}

\begin{abstract}
We prove a positive mass theorem for one-ended asymptotically flat spin
initial data with multiple time directions and a smooth compact boundary.
Assuming the trace-norm dominant energy condition and pairwise commutativity
of the second fundamental forms, a fixed auxiliary time direction on each
boundary component determines a local chirality condition. We identify the
surviving mixed boundary term and obtain the inequality
$E\geq\|\mathcal P\|_{\mathrm{tr}}$ under the boundary condition
$H+\operatorname{tr}_\Sigma k^T+\|\mathcal Q_T\|_{\mathrm{tr}}\leq0$.
The proof constructs a Dirac--Witten harmonic spinor in an affine energy
space and evaluates the mass identity by approximation. We also establish
the pointwise sharpness of this boundary condition for the fixed projector
and give two-time data with a nonzero mixed boundary term.
\end{abstract}

\maketitle

\tableofcontents

\newpage

\section{Introduction}
The positive mass theorem is one of the fundamental global results in
mathematical relativity.  An asymptotically flat initial data set consists of a
Riemannian manifold $(M^n,g)$ and a symmetric $(0,2)$-tensor $k$, interpreted as
the first and second fundamental forms of a spacelike hypersurface in a
Lorentzian spacetime $(\mathcal{M},\textbf{g})$. From the Einstein constraint equations, the corresponding energy and momentum densities are
\begin{equation}\label{eq:classical-constraints}
  \mu
  =\frac12\left(R_g+(\operatorname{tr}_g k)^2-|k|_g^2\right),
  \qquad
  J=\operatorname{div}_g\left(k-(\operatorname{tr}_g k)g\right).
\end{equation}
 In the spin setting, Witten's Dirac-operator argument gives
the following form of the spacetime positive mass theorem
\cite{Witten1981,DingSpacetimePMT}.
\begin{theorem} [Witten] Let $(M^n, g,k)$ be asymptotically flat spin initial data set satisfying the dominant energy condition
\[
\mu\geq |J|
\]
Then for any asymptotically flat ends, the $ADM$ energy $E$ and linear momentum $P$ satisfy
\[
E\geq |P|.
\]   
\end{theorem}
The proof uses a harmonic spinor with prescribed constant value at infinity.
An integrated Schr\"odinger--Lichnerowicz formula expresses a quadratic form in
$(E,P)$ as a bulk integral.  The dominant energy condition makes that integral
nonnegative, and a suitable choice of the asymptotic spinor yields the stated
inequality.

Hirsch, Payne, and Zhang recently extended this spinorial framework to
generalized initial data with several timelike directions
\cite[Theorem~1.2]{HirschPayneZhang2026}. Such initial data set consist of
\[
  (M^n,g,k^1,\ldots,k^m),\qquad 1\leq m\leq n,
\]
where $k^1,\ldots,k^m$ are symmetric $(0,2)$-tensors.  If $M$ is realized as a
spacelike submanifold of a pseudo-Riemannian manifold of signature $(n,m)$, then these tensors are the components of its vector-valued second fundamental
form.  The generalized constraint quantities are
\begin{equation}\label{eq:generalized-constraints}
  \mu
  =\frac12\left[
    R_g+\sum_{\alpha=1}^m
    \left((\operatorname{tr}_g k^\alpha)^2-|k^\alpha|_g^2\right)
  \right],
  \qquad
  J^\alpha
  =\operatorname{div}_g\left(k^\alpha-(\operatorname{tr}_g k^\alpha)g\right).
\end{equation}
Let $\mathcal J$ be the $m\times n$ matrix whose rows are
$J^1,\ldots,J^m$, and let $\mathcal P$ be the $m\times n$ matrix whose rows
are the ADM momentum vectors $P^1,\ldots,P^m$.  If
$\|\cdot\|_{\mathrm{tr}}$ denotes the trace norm,  namely the sum of the
singular values, then the multiple-time dominant energy condition is
\begin{equation}\label{eq:multiple-time-dec}
  \mu\geq \|\mathcal J\|_{\mathrm{tr}}.
\end{equation}
The main inequality of Hirsch--Payne--Zhang is as follows:
\begin{theorem} [Hirsch--Payne--Zhang] \label{thm:hpz} Let $(M^n, g, k^1,\dots,k^m)$, with $m\leq n$, be an asymptotically flat spin initial data set. Assume that the dominant energy condition
\[
\mu \geq \lVert \mathcal{J} \rVert_{\mathrm{tr}}
\]
holds, and that the \(k^\alpha\), \(\alpha=1,\ldots,m\), pairwise commute
as \(g\)-self-adjoint endomorphisms. Then the energy--momentum satisfies
\[
E \geq \lVert \mathcal{P} \rVert_{\mathrm{tr}}.
\]    
\end{theorem}
As Hirsch, Payne, and Zhang point out, these generalized quantities are defined mathematically, rather than being
derived from a proposed system of Einstein equations with several time
directions.  When $m=1$, they reduce to the usual constraint and ADM
quantities.  Pairwise commutativity is used in their generalized
Schr\"odinger--Lichnerowicz formula to eliminate mixed Clifford-curvature
terms.

Earlier, Chen--Hijazi--Zhang studied the Dirac--Witten operator on
spacelike submanifolds of pseudo-Riemannian manifolds
\cite{ChenHijaziZhang2012}. Their Theorem~4.1 includes generalized
apparent horizons of future or past type, assumes a spin normal bundle of
odd rank, and states a bound by the Euclidean norm of all momentum
components. Their boundary chirality uses the product of the normal time
Clifford generators. The present result instead uses the abstract
Hirsch--Payne--Zhang data and trace-norm energy condition, allows both odd
and even $m\leq n$, and fixes one time direction at each boundary component.
Its additional boundary requirement controls the surviving $T^\perp$
mixed term. These differences in the hypotheses and boundary conditions
preclude identifying the two results. We do not use the earlier proof;
for the algebraic issues in that argument, see
\cite[Remark~3.7]{HirschPayneZhang2026}.

The goal of the present paper is to develop a compact-boundary analogue of
Theorem \ref{thm:hpz}. Recall that compact inner boundaries model black-hole regions in
the ordinary one-time setting.  Let $\Sigma=\partial M$, and let $\nu$ be the
unit normal pointing from $\Sigma$ into $M$ toward the designated
asymptotically flat end. 
For a classical initial data set $(M,g,k)$, we define the outward future null
expansion by
\begin{equation}\label{eq:theta-plus}
  \theta_+=H+\operatorname{tr}_\Sigma k,
  \qquad
  \operatorname{tr}_\Sigma k=\sum_{A=1}^{n-1}k(e_A,e_A),
\end{equation}
where $H$ is the mean curvature of $\Sigma$ with respect to $\nu$ and
$e_1,\ldots,e_{n-1}$ is a local orthonormal frame of $T\Sigma$. Thus,  
\begin{definition}
Given a spacelike surface $\Sigma^{n-1}$, in either a spacetime
$(\mathcal{M},\mathbf{g})$ or an initial data set $(M^n,g,k)$, we say that
$\Sigma$ is
\[
  \begin{array}{c|c}
    \text{outer trapped} & \text{if}\ \theta_+<0,\\
    \text{weakly outer trapped} & \text{if}\ \theta_+\leq0,\\
    \text{marginally outer trapped (a MOTS).}  & \text{if} \ \theta_+=0
  \end{array}
\]
\end{definition}
The spinorial positive mass theorem with an inner black-hole boundary
originates in work of Gibbons, Hawking, Horowitz, and Perry
\cite{GibbonsHawkingHorowitzPerry1983}.  Herzlich supplied a rigorous
treatment of the trapped-boundary argument \cite{Herzlich1998}, and Bartnik
and Chru\'sciel developed the corresponding Fredholm theory for Dirac-type
boundary problems \cite{BartnikChrusciel2005}.  These results give the
following one-time statement; see also \cite{GallowayLee2022,HirschKazarasKhuri2022}
for related boundary formulations and alternative methods.

\begin{theorem}[One-time boundary positive mass theorem]  \label{thm:one-time-boundary}
  Let $(M^n,g,k)$ be an asymptotically flat spin initial data set with smooth
  compact boundary $\Sigma$.  Suppose that the dominant energy condition
  holds and that $\Sigma$ is weakly trapped, ie. 
  \[
    H+\operatorname{tr}_\Sigma k\leq0.
  \]
  Then the ADM energy--momentum of the designated asymptotically flat end
  satisfies
  \[
    E\geq |P|.
  \]
\end{theorem}
For $m>1$, note that the tangential traces
\[
  \tau^\alpha=\operatorname{tr}_\Sigma k^\alpha,
  \qquad \alpha=1,\ldots,m,
\]
form a vector in the auxiliary $m$-dimensional timelike space $\bigl(\mathcal{T},\langle\cdot,\cdot\rangle_{\mathcal{T}}\bigr)
\cong \bigl(\mathbb{R}^m,\delta\bigr)$. In the ordinary $m=1$ Lorentzian case, the time orientation selects a unit time directions 
from two unit time directions, future and past. However, for $m>1$, there is no distinguished future component of the unit time sphere. Therefore, a preferred unit timelike direction must be included in the boundary data if one wants a one-sided null expansion.

Now fix an orthonormal reference basis $\{T_1,\ldots,T_m\}$ of the auxiliary timelike space $\mathcal T$ and consider the generalized second fundamental form as the
$\mathcal{T}$-valued symmetric tensor $k=\sum_{\alpha=1}^{m} k^\alpha\otimes T_\alpha$.  Note that the basis  vectors can subsequently be regarded as constant sections of the trivial bundle $\Sigma\times\mathcal T\to \Sigma$ and then they form a parallel orthonormal frame along $\Sigma$. On each connected component of $\Sigma$, choose a unit vector so that 
\begin{equation}
T=\sum_{\alpha=1}^{m} t^\alpha T_\alpha,
\qquad
\sum_{\alpha=1}^{m}(t^\alpha)^2=1,
\end{equation}
and define 
\begin{equation}
k^T
:=
\langle \textbf{k},T\rangle_{\mathcal{T}}
=
\sum_{\alpha=1}^{m} t^\alpha k^\alpha,
\qquad
\theta_T
:=
H+\operatorname{tr}_{\Sigma}k^T.
\end{equation}
We refer to $\theta_T$ as the outward null expansion relative to $T$.  Thus a
$T$-MOTS satisfies $\theta_T=0$, while a weakly $T$-outer trapped boundary
satisfies $\theta_T\leq0$. But this condition by itself does not yet control the
full multiple-time spinorial boundary term. 

To see that, consider the vector-valued mixed component
$\textbf{k}(\nu,X)=\sum_{\alpha=1}^mk^{\alpha}(\nu,X)T_\alpha\in \mathcal T$ and extend $T$ to an orthonormal basis $\{T=\hat{T}_1,\hat{T}_2,\ldots,\hat{T}_m\}$ of $\mathcal{T}$ via Gram Schmidt. 
We define the map
\begin{equation}
 \mathcal Q_T:T\Sigma\longrightarrow T^\perp,
\qquad
\mathcal Q_T(X)
=
\operatorname{proj}_{T^\perp}\bigl(\textbf{k}(\nu,X)\bigr),
\end{equation}
where $T^\perp$ is the orthogonal complement of $T$ in $\mathcal{T}$. In the adapted frames, we have the 
$(m-1)\times(n-1)$ matrix
\begin{equation}\label{eq:boundary-mixed-matrix}
  (\mathcal Q_T)_{\beta A}
  =k^{\hat{T}_\beta}(\nu,e_A),
  \qquad
  2\leq\beta\leq m,
  \quad 1\leq A\leq n-1.
\end{equation}
Note that the pointwise trace norm  $\|\mathcal Q_T\|_{\mathrm{tr}}$ is independent of the chosen orthonormal basis of $T^\perp$ and local orthonormal frame of $T\Sigma$. 
We introduce the sufficient boundary sign condition

\begin{equation}\label{eq:adapted-trapping-condition}
  \Theta_{\Sigma,T}
  :=H+\operatorname{tr}_\Sigma k^T+\|\mathcal Q_T\|_{\mathrm{tr}}
  \leq0.
\end{equation}

The condition \eqref{eq:adapted-trapping-condition} is dictated by the
boundary Clifford algebra rather than by a proposed causal theory with
several time directions.  The proof uses the local spinor projector
$\nu\eta_T\psi=\psi$ and shows that the only additional boundary
endomorphism is controlled by $\|\mathcal Q_T\|_{\mathrm{tr}}$. Note that the spinor
projector is an analytic boundary condition and when
$m=1$, $\mathcal Q_T$ is empty and
\eqref{eq:adapted-trapping-condition} is exactly the ordinary weakly outer
trapped condition.

Our main result is the following boundary extension of the
Hirsch--Payne--Zhang inequality.
\begin{theorem}[Multiple-time positive mass theorem with boundary]
  \label{thm:main-boundary}
  Let $(M^n,g,k^1,\ldots,k^m)$, with $m\leq n$, be a smooth, connected, one-ended,
  asymptotically flat spin generalized initial data set with smooth compact
  boundary, complete up to the boundary.  Assume the decay and integrability hypotheses of
  Definition~\ref{def:af-data} as well as the dominant energy condition 
  \[
    \mu\geq\|\mathcal J\|_{\mathrm{tr}}
  \]
  and that the $k^\alpha$ commute pairwise as $g$-self-adjoint
  endomorphisms.  Suppose that on each connected component
  $\Sigma\subset\partial M$ there is a fixed unit time direction $T$ for
  which $T$ is constant in the auxiliary flat time bundle and the adapted
  trapping condition
  \[
    H+\operatorname{tr}_\Sigma k^T+\|\mathcal Q_T\|_{\mathrm{tr}}\leq0
  \]
  holds.  Then
  \[
    E\geq\|\mathcal P\|_{\mathrm{tr}}.
  \]
\end{theorem}
If $\partial M=\varnothing$, then it specializes, under the stronger one-ended regularity and integrability hypotheses imposed here, to the inequality of
\cite[Theorem~1.2]{HirschPayneZhang2026}; and if $m=1$, then it reduces to the
one-time weakly outer trapped boundary theorem.

We briefly describe the proof strategy.  On the enlarged spinor bundle of
Hirsch--Payne--Zhang, the weighted
boundary theory of Bartnik--Chruściel produces a Dirac--Witten harmonic
spinor in the affine energy class determined by a constant spinor
$\psi_{\infty}$, subject to the $T$-chirality condition at $\Sigma$.
The generalized Schrödinger--Lichnerowicz identity, using the pairwise
commutativity of the $k^\alpha$, has a nonnegative bulk term by the
dominant energy condition. The adapted trapping condition makes the
compact-boundary term nonnegative, with the surviving
$T^\perp$-mixed contribution controlled by
$\lVert \mathcal Q_T\rVert_{\mathrm{tr}}$. Evaluating the mass functional at
infinity and choosing $\psi_{\infty}\neq 0$ by a singular value
decomposition of $\mathcal{P}$ then gives
\[
\bigl(E-\lVert\mathcal{P}\rVert_{\mathrm{tr}}\bigr)
N(\psi_{\infty})\geq 0,
\]
and hence
\[
E\geq \lVert\mathcal{P}\rVert_{\mathrm{tr}}.
\]
The selection of a single direction $T$ is essential because the chirality
involutions associated with distinct orthogonal time directions
anticommute.

The paper is organized as follows.  Section~\ref{sec:preliminaries} fixes the
asymptotic, weighted-space, and Clifford conventions.  Section~\ref{sec:witten-identity}
proves the generalized Witten identity.  Sections~\ref{sec:boundary-geometry}
and~\ref{sec:boundary-current} explain Lee's local boundary projector \cite[Theorem~8.29]{LeeGeometricRelativity2023} and
compute the compact-boundary term.  Section~\ref{sec:boundary-problem}
checks the hypotheses of the weighted elliptic boundary theory.  Finally,
Section~\ref{sec:positive-mass} constructs the Witten spinor, evaluates the
mass boundary functional, and proves the theorem.
\section*{Acknowledgements} The author would like to thank Sven Hirsch for suggesting this problem and for his guidance throughout the development of this work.

\section{Preliminaries}\label{sec:preliminaries}
Throughout, $n\geq3$, $1\leq m\leq n$, and
$(M^n,g,k^1,\ldots,k^m)$ is a spin generalized initial data set with smooth compact boundary $\partial M=\Sigma$. We assume that $M$ is connected and complete up to this boundary: the length metric on $M$ including $\Sigma$ is complete. Spatial
indices $i,j,\ell$ range from $1$ to $n$, boundary indices $A,B$ from $1$ to $n-1$, and time indices
$\alpha,\beta$ range from $1$ to $m$.  Repeated indices are summed. We take the real part of the Hermitian spinor pairing and denote it by $(,)$
so that all pointwise quadratic expressions below are real.

\subsection{Asymptotic and weighted-space conventions}

\begin{definition}[Asymptotically flat data]\label{def:af-data}
  We say that $(M^n,g,k^1,\ldots,k^m)$ is \emph{asymptotically flat of rate $\tau$} if there is a compact $K\subset M$ such that $M\setminus K$ is diffeomorphic to $\mathbb{R}^n\setminus B_{R}$, In the resulting asymptotic coordinate $x=(x^1,\dots x^n)$ with $r=|x|$, 
  \begin{equation}\label{eq:af-decay}
    g_{ij}-\delta_{ij}=O_{2,\gamma}(r^{-\tau}),
    \qquad
    k^\alpha_{ij}=O_{1,\gamma}(r^{-\tau-1}),
  \end{equation}
where $0<\gamma<1$ and $\tau>(n-2)/2$.
\end{definition}
Recall that we define the energy density $\mu$ and momentum densities $J^\alpha$ as 
\begin{equation}\label{eq:constraint-definitions}
  \mu=\frac12\left[
    R_g+\sum_{\alpha=1}^m
    \big((\operatorname{tr}_gk^\alpha)^2-|k^\alpha|_g^2\big)
  \right],
  \qquad
  J^\alpha=\operatorname{div}_g\!\left(
    k^\alpha-(\operatorname{tr}_gk^\alpha)g
  \right).
\end{equation}
Let $\mathcal J=(J_i^\alpha)$ be the $m\times n$ matrix field and assume
the multiple-time dominant energy condition holds 
\begin{equation}\label{eq:dec-prelim}
  \mu\geq\|\mathcal J\|_{\mathrm{tr}}.
\end{equation}
In additon, we assume that $\mu, |J^\alpha|_g\in L^1(M)$ for every $\alpha$. The ADM energy and momentum matrices are normalized by
\begin{align}
 E&=\frac{1}{2(n-1)\omega_{n-1}}
   \lim_{r\to\infty}\int_{S_r}
   (\partial_i g_{ij}-\partial_jg_{ii})\nu_r^j\,dA,
   \label{eq:adm-energy-prelim}\\
 P^\alpha_i&=\frac{1}{(n-1)\omega_{n-1}}
   \lim_{r\to\infty}\int_{S_r}
   \big(k^\alpha_{ij}-(\operatorname{tr}_gk^\alpha)g_{ij}\big)
   \nu_r^j\,dA, 
   \label{eq:adm-momentum-prelim}
\end{align}
where the matrix with rows $P^\alpha$ is denoted by $\mathcal P$.

We also record the exact weighted convention needed in Section~\ref{sec:boundary-problem}.
For $s\in\mathbb R$, define
\begin{equation}\label{eq:weighted-L2}
  \|u\|_{L^2_s}^2
  :=\int_M |u|^2(1+r)^{-2s-n}\,dV_g,
  \qquad
  \|u\|_{W^{1,2}_s}^2
  :=\|u\|_{L^2_s}^2+\|\nabla u\|_{L^2_{s-1}}^2.
\end{equation}
Set the critical decay exponent and corresponding weighted index to be
\begin{equation}\label{eq:critical-weight}
  q_*:=\frac{n-2}{2},
  \qquad
  \delta_*:=-q_*=\frac{2-n}{2}.
\end{equation}
Then at this value, 
\begin{equation}\label{eq:critical-weight-expanded}
  \|u\|_{L^2_{\delta_*}}^2=
       \int_M(1+r)^{-2}|u|^2\,dV_g,
  \qquad
  \|\nabla u\|_{L^2_{\delta_*-1}}^2=\|\nabla u\|_{L^2}^2.
\end{equation}
Thus $W^{1,2}_{\delta_*}$ is the critical weighted energy space and this implies no pointwise decay rate.
\subsection{The enlarged Clifford module}
 
Let $S(M^n) \to M$ be the complex spinor bundle of $(M,g)$ and take 
 $\overline S(M^n):=S(M^n)^{2^m}$.
Following \cite[Section~2.2]{HirschPayneZhang2026}, $\overline S(M^n)$ carries
spatial Clifford actions and auxiliary time actions with the corresponding orthonormal basis $\{e_1,\dots,e_n,\eta_{1},\dots\eta_{m}\}$ of $\mathbb{R}^{n,m}$ satisfying the Clifford relations
\begin{equation}\label{eq:clifford-relations-prelim}
 \begin{split}
  e_ie_j+e_je_i&=-2\delta_{ij},\\
  \eta_\alpha\eta_\beta+\eta_\beta\eta_\alpha&=2\delta_{\alpha\beta},\\
  e_i\eta_\alpha+\eta_\alpha e_i&=0,
 \end{split}
 \qquad
  e_i^*=-e_i,
  \quad
  \eta_\alpha^*=\eta_\alpha.
\end{equation}
The enlarged bundle carries the direct-sum Hermitian product 
\[
\left\langle
(\psi_a)_{a=1}^{2^m},(\phi_a)_{a=1}^{2^m}
\right\rangle
=
\sum_{a=1}^{2^m}\langle \psi_a,\phi_a\rangle.
\]
For a vector $T=\sum_\alpha t^\alpha T_\alpha$ in $\mathcal{T}$, write
$\eta_T:=\sum_\alpha t^\alpha\eta_\alpha$ and for a spinor $\psi\in\Gamma(\overline S(M^n))$, we define the associated scalar field $N(\psi)$ and vector field $X^\alpha(\psi)$ on $M$ by 
\begin{equation}
 N(\psi):=|\psi|^2, \quad  \bigl\langle X_\alpha(\psi),Y\bigr\rangle
:=
\bigl\langle Y\eta_\alpha\psi,\psi\bigr\rangle, 
\end{equation}
for every vectorfield $Y$ on $M$. In a local orthonormal frame, 
\begin{equation}\label{eq:spinor-bilinears-prelim}
  X_i^\alpha(\psi):=(e_i\eta_\alpha\psi,\psi).
\end{equation}

\begin{lemma}\label{lem:trace-norm-clifford}
  For every real $m\times n$ matrix $A=(A_{\alpha i})$ and spinor
  $\psi\in \overline{S}(M^n)_x$, we have the estimate 
  \begin{equation}\label{eq:matrix-clifford-estimate}
  \left|
\left(
\sum_{\alpha=1}^{m}\sum_{i=1}^{n}
A_{\alpha i}e_i\eta_\alpha\psi,\psi
\right)
\right|
\le
\|A\|_{\mathrm{tr}}|\psi|^2, 
  \end{equation}
where 
\[\|A\|_{\mathrm{tr}}=
  \sum_{s=1}^{\min\{m,n\}}\sigma_s(A).\]
In particular,
\[
\left(
\sum_{\alpha=1}^{m}\sum_{i=1}^{n}
A_{\alpha i}e_i\eta_\alpha\psi,\psi
\right)
\ge
-\|A\|_{\mathrm{tr}}|\psi|^2.
\]
  Consequently, the dominant energy condition \eqref{eq:dec-prelim} implies that 
  \begin{equation}\label{eq:bulk-positivity}
    \mu N(\psi)+\sum_\alpha
    \langle J^\alpha,X^\alpha(\psi)\rangle\geq0.
  \end{equation}
\end{lemma}
For $A=\mathcal J$, this is precisely the estimate in
\cite[Lemma~3.5, equation~(3.4)]{HirschPayneZhang2026}.  The more
general matrix formulation isolates the Clifford-algebra fact that will also
be used for the compact-boundary matrix $\mathcal Q_T$ and the ADM momentum
matrix $\mathcal P$.

\begin{proof}
Let $r:=\min\{m,n\}$ and choose a singular-value decomposition $A=U^{T}\Lambda V$, with  $U\in O(m)$
$V\in O(n)$, and $\Lambda_{ss}=\sigma_s(A)$. Define the rotated Clifford generators
  \[
    \widehat\eta_s:=\sum_\alpha U_{s\alpha}\eta_\alpha,
    \qquad
    \widehat e_s:=\sum_iV_{si}e_i.
  \]
  Orthogonality of $U,V$ and the Clifford relations (\ref{eq:clifford-relations-prelim}) give
  \[
   \widehat e_s^2=-I,
   \quad \widehat\eta_s^2=I,
   \quad \widehat e_s\widehat\eta_s=-\widehat\eta_s\widehat e_s,
   \quad \widehat e_s^*=-\widehat e_s,
   \quad \widehat\eta_s^*=\widehat\eta_s.
  \]
  Hence $G_s:=\widehat e_s\widehat\eta_s$ satisfies,
  \[
    G_s^*=\widehat\eta_s(-\widehat e_s)
         =-\widehat\eta_s\widehat e_s
         =\widehat e_s\widehat\eta_s=G_s,
    \qquad
    G_s^*G_s=G_s^2=-\widehat e_s^{2}\widehat\eta_s^{2}=I.
  \]
It follows that 
\begin{align*}
|G_s\psi|^2=\langle\psi,G_s^*G_s\psi\rangle= |\psi|^2 \Rightarrow |G_s\psi|=|\psi|
\end{align*}
Then we have 
\begin{align*}
\bigl|(G_s\psi,\psi)\bigr|
&\leq
\bigl|\langle G_s\psi,\psi\rangle\bigr| \\
&\leq
|G_s\psi|\,|\psi| \\
&=
|\psi|^2.
\end{align*}
Writing  \(A_{\alpha i}\) in components, we obtain
\begin{align*}
\left|\left(\sum_{\alpha,i}
A_{\alpha i}e_i\eta_\alpha\psi,
\psi \right)\right|
&=
\left|\left(\sum_{s=1}^{r}\sigma_s(A)
\sum_{\alpha,i}
V_{si}U_{s\alpha}e_i\eta_\alpha\psi,\psi\right)\right| \\
&=
\left|\left(\sum_{s=1}^{r}\sigma_s(A)
\hat{e}_s\hat{\eta}_s\psi,\psi \right)\right| \\
&=
\left|\left(\sum_{s=1}^{r}\sigma_s(A)G_s\psi,\psi \right)\right|\\
&\leq
\sum_{s=1}^{r}
\sigma_s(A)\bigl|(G_s\psi,\psi)\bigr|\\
&\leq
\sum_{s=1}^{r}
\sigma_s(A)|\psi|^2\\
&=\lVert A\rVert_{\mathrm{tr}}|\psi|^2.
\end{align*}

Thus,
\[
\left|
\left(
\sum_{\alpha,i}
A_{\alpha i}e_i\eta_\alpha\psi,\psi
\right)
\right|
\leq
\lVert A\rVert_{\mathrm{tr}}|\psi|^2.
\]

Now to show  \eqref{eq:bulk-positivity}, in an orthonormal frame,
\begin{align*}
\sum_{\alpha=1}^{m}
\langle J^\alpha,X^\alpha(\psi)\rangle
&=
\sum_{\alpha=1}^{m}\sum_{i=1}^{n}
J_i^\alpha X_i^\alpha(\psi) \\
&=
\left(
J_i^\alpha e_i\eta_\alpha\psi,\psi
\right)\\
&\geq
-\lVert J\rVert_{\mathrm{tr}}\,|\psi|^2,
\end{align*}
where we use \eqref{eq:matrix-clifford-estimate} in the last inequality. It follows from the multiple-time dominant energy condition  that 
\begin{align*}
\mu N(\psi)
+
\sum_{\alpha}
\langle J^\alpha,X^\alpha(\psi)\rangle
&\geq
\mu|\psi|^2
-
\lVert J\rVert_{\mathrm{tr}}\,|\psi|^2 \\
&=
\bigl(\mu-\lVert J\rVert_{\mathrm{tr}}\bigr)|\psi|^2\\
&\geq 0.
\end{align*}
\end{proof}
\section{The generalized Witten identity}\label{sec:witten-identity}

Let $\nabla$ be the spin connection and let $D=e_i\nabla_i$ be the ordinary Dirac
operator on $\overline{S}(M^n)$.  Define
\begin{equation}\label{eq:modified-connection-prelim}
  \overline\nabla_i\psi
  :=\nabla_i\psi+\frac12\sum_\alpha k^\alpha_{ij}e_j\eta_\alpha\psi,
  \qquad
  \overline D:=e_i\overline\nabla_i.
\end{equation}
and note that the endomorphisms $\eta_\alpha$ are parallel, ie. $\nabla_i\eta_\alpha=0$. 

Since $k^\alpha$ is symmetric, we have 
\begin{align}
 \overline D
 &=D+\frac12\sum_{\alpha,i,j}k^\alpha_{ij}e_ie_j\eta_\alpha\notag\\
 &=D+\frac14\sum_{\alpha,i,j}k^\alpha_{ij}(e_ie_j+e_je_i)\eta_\alpha\notag\\
 &=D-\frac12\sum_\alpha\operatorname{tr}_gk^\alpha\eta_\alpha.
 \label{eq:modified-dirac-zero-order}
\end{align}
We set  $\pi^\alpha_{ij}:=k^\alpha_{ij}-\operatorname{tr}_gk^\alpha g_{ij}$ and define the Witten current
\begin{equation}\label{eq:corrected-witten-current}
  \mathcal U_i(\psi)
  :=(e_iD\psi+\nabla_i\psi,\psi)
    +\frac12\sum_\alpha
      \pi^\alpha_{ij}(e_j\eta_\alpha\psi,\psi).
\end{equation}
Equivalently, direct use of gives 
\eqref{eq:modified-connection-prelim}--\eqref{eq:modified-dirac-zero-order}
\begin{align*}
e_i\overline D+\overline \nabla_i
&=e_i\left(D-\frac12 \operatorname{tr}_gk^\alpha\eta_\alpha\right)
  +\nabla_i+\frac12k_{ij}^{\alpha}e_j\eta_\alpha \\
&=e_iD+\nabla_i
  +\frac12\bigl(k_{ij}^{\alpha}e_j-\operatorname{tr}_gk^\alpha e_i\bigr)\eta_\alpha\\
&= e_iD+\nabla_i
  +\frac12 \bigl(k_{ij}^{\alpha}-\operatorname{tr}_gk^\alpha\delta_{ij}\bigr)e_j\eta_\alpha\\
&= e_iD+\nabla_i
+\frac12\pi_{ij}^{\alpha}e_j\eta_\alpha. 
\end{align*}
Thus, we have the following normalization 
\begin{equation}\label{eq:current-modified-form}
  \mathcal U_i(\psi)
   =(e_i\overline D\psi+\overline\nabla_i\psi,\psi).
\end{equation}

\begin{proposition}[Generalized Witten identity]\label{prop:witten-identity}
  If the $k^\alpha$, viewed as $g$-self-adjoint endomorphisms, commute
  pairwise, then
  \begin{equation}\label{eq:pointwise-witten-identity}
   \nabla_i\mathcal U_i(\psi)
   =|\overline\nabla\psi|^2-|\overline D\psi|^2
    +\frac12\mu|\psi|^2
    +\frac12\sum_\alpha
      \langle J^\alpha,X^\alpha(\psi)\rangle.
  \end{equation}
\end{proposition}

\begin{proof}
  Fix $p\in M$ and choose a local orthonormal frame with
  $\nabla_{e_i} e_j=0$ at $p$ and write $\mathcal U_i=\mathcal U_i^0+\mathcal U_i^1$, where 
  \begin{equation}
  \mathcal U_i^0:=\bigl(e_iD\psi+\nabla_i\psi,\psi\bigr), \quad \mathcal U_i^1:=\frac12\pi_{ij}^{\alpha}
\bigl(e_j\eta_\alpha\psi,\psi\bigr)    
  \end{equation}

  For
  $\mathcal U_i^0=(e_iD\psi+\nabla_i\psi,\psi)$, differentiation gives
  \begin{align*}
   \nabla_i\mathcal U_i^0
   &=(e_i\nabla_iD\psi+\nabla_i\nabla_i\psi,\psi)
     +(e_iD\psi+\nabla_i\psi,\nabla_i\psi).
  \end{align*}
  By skew-adjointness of $e_i$, the second term writes 
  \begin{align*}
    (e_iD\psi+\nabla_i\psi,\nabla_i\psi)
      &=-(D\psi,e_i\nabla_i\psi)+\bigl(\nabla_i\psi,\nabla_i\psi\bigr)\\
      &=-|D\psi|^2+|\nabla\psi|^2.
\end{align*}
We also have 
\begin{align*}
(e_i\nabla_iD\psi+\nabla_i\nabla_i\psi,\psi)
&=(D^2\psi-\nabla^*\nabla\psi,\psi)\\
&=\frac14R_g|\psi|^2.
\end{align*}
where we used the Schr\"odinger--Lichnerowicz formula
  $D^2=\nabla^*\nabla+R_g/4$ in the last equality. Thus, combining both expressions above yields
  \begin{equation}\label{eq:ordinary-witten-divergence}
    \nabla_i\mathcal U_i^0
      =|\nabla\psi|^2-|D\psi|^2+\frac14R_g|\psi|^2.
  \end{equation}
  For the momentum part
  $\mathcal U_i^1=\frac12
  \pi^\alpha_{ij}(e_j\eta_\alpha\psi,\psi)$, note that
  $(e_j\eta_\alpha)^*=e_j\eta_\alpha$.  Hence, at the normal-frame point $p$ and using $\nabla_i\eta_\alpha=0$, 
  \begin{align}
   \nabla_i\mathcal U_i^1
   &=\frac12 (\nabla_i\pi^\alpha_{ij})
       (e_j\eta_\alpha\psi,\psi)\notag +\frac12\sum_\alpha\pi^\alpha_{ij}
       \big((e_j\eta_\alpha\nabla_i\psi,\psi)
           +(e_j\eta_\alpha\psi,\nabla_i\psi)\big)\notag\\
   &=\frac12(\nabla_i\pi_{ij}^{\alpha})
 \bigl(e_j\eta_\alpha\psi,\psi\bigr)
+\pi_{ij}^{\alpha}
 \bigl(e_j\eta_\alpha\nabla_i\psi,\psi\bigr) \notag\\
   &=\frac12 J_j^\alpha
       (e_j\eta_\alpha\psi,\psi)
     +\pi^\alpha_{ij}
       (e_j\eta_\alpha\nabla_i\psi,\psi).
   \label{eq:momentum-current-divergence}
  \end{align}
  To identify the second term, write
  \[
    A_i:=\frac12 k^\alpha_{ij}e_j\eta_\alpha,
    \qquad
    \Phi:=-\frac12 \operatorname{tr}_gk^\alpha\eta_\alpha,
  \]
  so $\overline\nabla_i=\nabla_i+A_i$ and
  $\overline D=D+\Phi$.  Direct expansion gives
  \begin{align}
   |\overline\nabla\psi|^2-|\overline D\psi|^2
   &=|\nabla\psi|^2-|D\psi|^2
     +2(A_i\psi,\nabla_i\psi)-2(\Phi\psi,D\psi) +\sum_i|A_i\psi|^2-|\Phi\psi|^2.
   \label{eq:modified-square-expansion}
  \end{align}
  Using $(\eta_\alpha\psi,e_i\nabla_i\psi)
  =-(e_i\eta_\alpha\nabla_i\psi,\psi)$, the first-order part is
  \begin{align}
2(A_i\psi,\nabla_i\psi)-2(\Phi\psi,D\psi)&=k_{ij}^{\alpha}
\bigl(e_j\eta_\alpha\psi,\nabla_i\psi\bigr)+\operatorname{tr}_gk^\alpha\bigl(\eta_\alpha\psi,e_i\nabla_i\psi\bigr)\notag\\
&=k^\alpha_{ij}
          (e_j\eta_\alpha\nabla_i\psi,\psi)
       -\operatorname{tr}_gk^\alpha
          (e_i\eta_\alpha\nabla_i\psi,\psi)\notag\\
&=\bigl(k_{ij}^{\alpha}-\operatorname{tr}_gk^\alpha\delta_{ij}\bigr)
 \bigl(e_j\eta_\alpha\nabla_i\psi,\psi\bigr)\notag\\
&=\pi^\alpha_{ij}
          (e_j\eta_\alpha\nabla_i\psi,\psi).
   \label{eq:first-order-match}
  \end{align}
  Now from the Clifford relations,
  \begin{align}
   \sum_i|A_i\psi|^2=\bigl(A_i^2\psi,\psi\bigr)&=\big(\frac14k_{ij}^{\alpha}k_{i\ell}^{\beta}
 e_j\eta_\alpha e_\ell\eta_\beta\psi,\psi\big)\notag\\
 &=-\frac14k_{ij}^{\alpha}k_{i\ell}^{\beta}
 \bigl(e_je_\ell\eta_\alpha\eta_\beta\psi,\psi\bigr) \notag\\
   &=\frac14\sum_\alpha|k^\alpha|^2|\psi|^2
     -\frac14\sum_{\alpha<\beta}
       \big([k^\alpha,k^\beta]_{j\ell}
       e_je_\ell\eta_\alpha\eta_\beta\psi,\psi\big),
   \label{eq:A-square}
  \end{align}
  where $[k^\alpha,k^\beta]_{j\ell}
   =k^\alpha_{ij}k^\beta_{i\ell}
    -k^\beta_{ij}k^\alpha_{i\ell}=0$ by pairwise commutativity of $k^\alpha$. 
  Similarly, anticommutation of distinct time generators and $\Phi^*=\Phi$ give
  \begin{align*}
    \Phi^*\Phi=\Phi^2
    &=\frac14\sum_{\alpha=1}^m(\operatorname{tr}_gk^\alpha)^2I\\
    &\quad+\frac14\sum_{1\leq\alpha<\beta\leq m}
    \operatorname{tr}_gk^\alpha\operatorname{tr}_gk^\beta
    (\eta_\alpha\eta_\beta+\eta_\beta\eta_\alpha)\\
    &=\frac14\sum_\alpha(\operatorname{tr}_gk^\alpha)^2I.
  \end{align*}
  Pairing with $\psi$ therefore yields
  \begin{equation}\label{eq:Phi-square}
    |\Phi\psi|^2=(\Phi^2\psi,\psi)
    =\frac14\sum_\alpha(\operatorname{tr}_gk^\alpha)^2|\psi|^2.
  \end{equation}
  By pairwise commutativity of of $k^{\alpha}$, the final term of \eqref{eq:A-square}
  vanish.  Substituting \eqref{eq:first-order-match}--\eqref{eq:Phi-square}
  into \eqref{eq:modified-square-expansion}, we can rewrite \eqref{eq:ordinary-witten-divergence} as 
\begin{align} \label{eq:rewrite ordinary-witten-divergence}
\nabla_i\mathcal U_i^0
&=|\overline\nabla\psi|^2-|\overline D\psi|^2-\pi_{ij}^{\alpha}
 \bigl(e_j\eta_\alpha\nabla_i\psi,\psi\bigr)
-\frac14\sum_{\alpha=1}^m
 \bigl(|k^{\alpha}|^2-(\operatorname{tr}_gk^\alpha)^2\bigr)|\psi|^2+\frac14R_g|\psi|^2
\end{align}  
Adding 
  \eqref{eq:momentum-current-divergence} and \eqref{eq:rewrite ordinary-witten-divergence} and by definition of $\mu$ leaves
\begin{align*}
\nabla_i\mathcal U^i(\psi)
&=|\overline\nabla\psi|^2-|\overline D\psi|^2 
+\frac14\left[
 R_g+\sum_{\alpha=1}^m
 \bigl((\operatorname{tr}_gk^\alpha)^2-|k^{\alpha}|^2\bigr)
 \right]|\psi|^2 
+\frac12\sum_{\alpha=1}^m
 \langle J^{\alpha},X^{\alpha}(\psi)\rangle\\
&= |\overline \nabla\psi|^2-|\overline D\psi|^2
+\frac12\mu|\psi|^2
+\frac12\sum_{\alpha=1}^m
 \langle J^{\alpha},X^{\alpha}(\psi)\rangle.
\end{align*}
as desired. 
\end{proof}

\section{Boundary spin geometry and projector}
\label{sec:boundary-geometry}

\subsection{Meaning and ellipticity of the boundary condition}
Choose a local orthonormal frame $e_1,\ldots,e_{n-1}$ for $T\Sigma$ and set
$e_n=\nu$ with the mean-curvature
\begin{equation}\label{eq:H-convention}
  H=\sum_{A=1}^{n-1}g(\nabla_{e_A}\nu,e_A).
\end{equation}
On the restricted spinor bundle define
\begin{equation}\label{eq:boundary-dirac}
  D_\Sigma:=\sum_{A=1}^{n-1}\nu e_A\nabla_{e_A}+\frac12H.
\end{equation}
Equivalently, $D_\Sigma$ is the intrinsic Dirac operator for the induced
boundary spin connection, and the spinorial Gauss formula is
\begin{equation}\label{eq:spinorial-gauss}
  \sum_A\nu e_A\nabla_{e_A}=D_\Sigma-\frac12H.
\end{equation}
In our multiple-time module, define
\begin{equation}\label{eq:boundary-involution}
  \varepsilon_T:=\nu\eta_T,
  \qquad
  P_{T,\pm}:=\frac12(I\pm\varepsilon_T).
\end{equation}
The local spinor boundary condition is
\begin{equation}\label{eq:chirality-boundary-condition}
  P_{T,-}\gamma\psi=0,
  \qquad\text{equivalently}\qquad
  \varepsilon_T\gamma\psi=\gamma\psi,
\end{equation}
where $\gamma:W^{1,2}_{\mathrm{loc}}(M)\to H^{1/2}(\Sigma)$ is the trace
map.  The equation is imposed in $H^{1/2}(\Sigma)$ for an energy solution, and it holds pointwise after boundary regularity. 

\begin{lemma}
On each connected component of $\Sigma$, let $T$ be a unit vector parallel in
the flat auxiliary time bundle, and put
\[
 \varepsilon_T=\nu\eta_T,
 \qquad P_{T,\pm}=\frac12(I\pm\varepsilon_T).
\]
Then $\varepsilon_T$ is a self-adjoint involution, and
\begin{equation}\label{eq:4.6}
 D_\Sigma\varepsilon_T=-\varepsilon_TD_\Sigma.
\end{equation}
The local boundary condition $P_{T,-}\gamma u=0$ is elliptic for $D$ and
$\overline D$. Its Green adjoint boundary condition is again
$P_{T,-}\gamma u=0$.
\end{lemma}

\begin{proof}
Since  $T$ has unit length, the Clifford relations give
$\eta_T^2=I$, $\eta_T^*=\eta_T$, $\nu^2=-I$, $\nu^*=-\nu$, and
$\nu\eta_T=-\eta_T\nu$.  Consequently,
\[
 \varepsilon_T^*
 =\eta_T^*\nu^*
 =-\eta_T\nu
 =\nu\eta_T
 =\varepsilon_T,
 \qquad
 \varepsilon_T^2
 =\nu\eta_T\nu\eta_T
 =-\nu^2\eta_T^2=I.
\]
Thus $P_{T,\pm}=\tfrac12(I\pm\varepsilon_T)$ are orthogonal projections:
\[
 P_{T,\pm}^*=P_{T,\pm},\qquad
 P_{T,\pm}^2=P_{T,\pm},\qquad
 P_{T,+}P_{T,-}=0,\qquad P_{T,+}+P_{T,-}=I.
\]
Moreover,
\[
 \varepsilon_T\nu=\eta_T,
 \qquad \nu\varepsilon_T=-\eta_T,
 \qquad \varepsilon_T\nu=-\nu\varepsilon_T.
\]
Therefore $\nu$ maps $S_+$ to $S_-$ and $S_-$ to $S_+$.
Because $\nu^{-1}=-\nu$, these maps are isomorphisms.  The two eigenbundles
have equal rank, each one half of the rank of $\overline S|_\Sigma$.\\
For $(2)$, we write
\[
 h_{AB}:=g(\nabla_{e_A}\nu,e_B),\qquad H=\sum_Ah_{AA},
\]
and boundary spin connection, with exactly the sign in (4.2)--(4.3), is
\[
 \nabla^\Sigma_{e_A}
 :=\nabla_{e_A}+\frac12\sum_Bh_{AB}e_B\nu.
\]
To check that it preserves the normal Clifford action, we compute its induced
connection on endomorphisms:
\begin{align*}
 \nabla^\Sigma_{e_A}\nu
 &=\sum_Bh_{AB}e_B
   +\frac12\sum_Bh_{AB}[e_B\nu,\nu]\\
 &=\sum_Bh_{AB}e_B
   +\frac12\sum_Bh_{AB}(-e_B-e_B)\\
&=0,
\end{align*}
where  $[e_B\nu,\nu]=e_B\nu^2-\nu e_B\nu=-2e_B$.
Since $T$ is parallel in the auxiliary time bundle,
$\nabla\eta_T=0$. Note that the product $e_B\nu$ commutes with $\eta_T$ as 
$\eta_T$ anticommutes with each of its two factors.  Hence, we have 
\[
 \nabla^\Sigma\eta_T=0,
 \qquad \nabla^\Sigma\varepsilon_T=0.
\]
The tangential boundary Clifford action $\nu e_A$ satisfies
\[
 (\nu e_A)\varepsilon_T=e_A\eta_T,
 \qquad
 \varepsilon_T(\nu e_A)=\eta_Te_A=-e_A\eta_T, 
\]
with the corresponding boundary Dirac operator agrees with $D_\Sigma$
\begin{align*}
 \sum_A\nu e_A\nabla^\Sigma_{e_A}
 &=\sum_A\nu e_A\nabla_{e_A}
    +\frac12\sum_{A,B}h_{AB}\nu e_Ae_B\nu\\
 &=\sum_A\nu e_A\nabla_{e_A}+\frac12H\\
 &=D_\Sigma, 
\end{align*}
where the second equality, the off-diagonal terms cancel by
$h_{AB}=h_{BA}$ and $e_Ae_B+e_Be_A=0$ for $A\ne B$, and each diagonal product
$\nu e_A^2\nu$ equals $I$.  Therefore, for every smooth boundary spinor $u$,
\begin{align*}
 D_\Sigma(\varepsilon_Tu)
 &=\sum_A\nu e_A\nabla^\Sigma_{e_A}(\varepsilon_Tu)\\
 &=\sum_A\nu e_A\varepsilon_T\nabla^\Sigma_{e_A}u\\
 &=-\varepsilon_T\sum_A\nu e_A\nabla^\Sigma_{e_A}u\\
&=-\varepsilon_TD_\Sigma u.
\end{align*}
This proves \eqref{eq:4.6}.\\
For the 
ellipticity, 
fix a boundary point, freeze the principal coefficeints, and choose inward normal coordinate $t\ge0$ and tangential coordinates
$y^1,\ldots,y^{n-1}$ such that the principal part of $D$ is
\[
 D_0=\nu\partial_t+\sum_Ae_A\partial_{y^A}.
\]
For a nonzero real tangential covector $\xi=(\xi_A)$, set
$\xi^\sharp\!\cdot:=\sum_A\xi_Ae_A$.  A Fourier mode
$e^{\mathrm i y\cdot\xi}u(t)$ satisfies $D_0u=0$ exactly when
\[
 \nu u'(t)+\mathrm i\sum_A\xi_Ae_Au(t)=0,
 \qquad
 u'(t)+B(\xi)u(t)=0,
 \qquad
 B(\xi):=-\mathrm i\nu\sum_A\xi_Ae_A.
\]
Since $\nu\sum_A\xi_Ae_A$ is skew-adjoint and has square $-|\xi|^2I$,
\[
 B(\xi)^*=B(\xi),\qquad B(\xi)^2=|\xi|^2I.
\]
Thus the solutions decaying as $t\to+\infty$ are
\[
 u(t)=e^{-t|\xi|}u_0,
 \qquad u_0\in C_+(\xi):=\ker(B(\xi)-|\xi|I).
\]
The preceding Clifford calculation gives
$B(\xi)\varepsilon_T=-\varepsilon_TB(\xi)$.  Therefore
$\varepsilon_T$ interchanges the positive and negative eigenspaces of
$B(\xi)$, and each has half the total dimension.  If
$u_0\in C_+(\xi)$ also satisfies the homogeneous boundary condition
$\varepsilon_Tu_0=u_0$, then
\[
 |\xi|u_0=B(\xi)\varepsilon_Tu_0
 =-\varepsilon_TB(\xi)u_0=-|\xi|u_0.
\]
Hence $u_0=0$.  Equivalently, $P_{T,-}:C_+(\xi)\longrightarrow S_-$
is injective, and it is an isomorphism since both spaces have half the
total dimension.  This is Lopatinski-Shapiro complementing condition . \\
Finally,for compactly supported smooth spinors $u,v$, we havae Green's formula at the
compact boundary as 
\[
 \int_M\bigl(\langle Du,v\rangle-\langle u,Dv\rangle\bigr)\,dV_g
 =-\int_\Sigma\langle\nu u,v\rangle\,d\sigma.
\]
The zeroth-order term in $\overline D$ is
self-adjoint, so the same formula holds with $D$ replaced by
$\overline D$. If $\gamma u\in S_+$, then $\nu\gamma u\in S_-$.  Consequently the
boundary form vanishes for all admissible $u$ precisely when
\[
 \gamma v\perp\nu S_+=S_-.
\]
Since $S_-^\perp=S_+$, this is equivalent to
$P_{T,-}\gamma v=0$ as desired. 
\end{proof}

\subsection{The compact-boundary current}\label{sec:boundary-current}

\begin{proposition} \label{prop:boundary-current}
  Suppose $P_{T,-}\gamma\psi=0$.  With $\nu$ pointing toward infinity,
  \begin{equation}\label{eq:boundary-current-formula}
    \mathcal B_{\Sigma,T}(\psi):=\mathcal U_i(\psi)\nu^i
      =-\frac12(H+\operatorname{tr}_{\Sigma}k_T)|\psi|^2
       +\frac12(\mathcal A_T\psi,\psi),
  \end{equation}
  where
  \begin{equation}\label{eq:AT-def}
    \mathcal A_T:=\sum_{\beta=2}^m\sum_{A=1}^{n-1}
      (\mathcal Q_T)_{\beta A}e_A\eta_{T_\beta}.
  \end{equation}
  Furthermore,
  \begin{equation}\label{eq:boundary-current-lower-bound}
   \mathcal B_{\Sigma,T}(\psi)
   \geq-\frac12\left(
      H+\operatorname{tr}_\Sigma k^T+\|\mathcal Q_T\|_{\mathrm{tr}}
   \right)|\psi|^2.
  \end{equation}
\end{proposition}

\begin{proof}

Recall that $\pi^{T_\beta}
=
k^{T_\beta}
-
\bigl(\operatorname{tr}_{g}k^{T_\beta}\bigr)g$, and in the adapted time frame, the Witten current may be written as
\[
\mathcal U_i(\psi)
=
\bigl(e_iD\psi+\nabla_i\psi,\psi\bigr)
+
\frac{1}{2}
\sum_{\beta=1}^{m}
\pi_{ij}^{T_\beta}
\bigl(e_j\eta_{T_\beta}\psi,\psi\bigr).
\]
Split the ordinary Dirac operator as 
\[
D
=
\sum_{A=1}^{n-1}e_A\nabla_A+\nu\nabla_\nu.
\]
Consequently, since $\nu^2=-I$, 
\begin{align*}
\nu D\psi+\nabla_\nu\psi
&=
\nu\left(
\sum_A e_A\nabla_A\psi+\nu\nabla_\nu\psi
\right)
+\nabla_\nu\psi \\
&=
\sum_A\nu e_A\nabla_A\psi
+\nu^2\nabla_\nu\psi
+\nabla_\nu\psi\\
&=\sum_A\nu e_A\nabla_A\psi\\
&=
D_\Sigma\psi-\frac{1}{2}H\psi, 
\end{align*}
where the last equality follows by the spinoral Gauss formula. 
Thus, we find
\begin{align} 
\bigl(\nu D\psi+\nabla_\nu\psi,\psi\bigr)
=
\bigl(D_\Sigma\psi,\psi\bigr)
-
\frac{1}{2}H|\psi|^2. \label{eq:normal-dirac-current}
\end{align}
By Lemma 4.1, we have 
\[
D_\Sigma\varepsilon_T
=
-\varepsilon_TD_\Sigma, \quad \varepsilon_T^*=\varepsilon_T, \quad \varepsilon_T^2=I.
\]
Then anticommutation and $\varepsilon_T\psi=\psi$ gives
\begin{align*}
\bigl(D_\Sigma\psi,\psi\bigr)
&=
\bigl(D_\Sigma\varepsilon_T\psi,\varepsilon_T\psi\bigr)\\
&=
\bigl(\varepsilon_TD_\Sigma\varepsilon_T\psi,\psi\bigr)\\
&=\bigl(-\varepsilon_T^2D_\Sigma\psi,\psi\bigr)\\
&=-\bigl(D_\Sigma\psi,\psi\bigr), 
\end{align*}
so $\bigl(D_\Sigma\psi,\psi\bigr)=0$ and thus 
\[
\bigl(\nu D\psi+\nabla_\nu\psi,\psi\bigr)
=
-\frac{1}{2}H|\psi|^2.
\]
Next, we contract the momentum current with \(\nu^i\)
\begin{align*}
\frac{1}{2}
\sum_{\beta}
\pi_{ij}^{T_\beta}\nu^i
\bigl(e_j\eta_{T_\beta}\psi,\psi\bigr)
&=
\frac{1}{2}
\sum_{\beta,A}
\pi^{T_\beta}(\nu,e_A)
\bigl(e_A\eta_{T_\beta}\psi,\psi\bigr)
+
\frac{1}{2}
\sum_{\beta}
\pi^{T_\beta}(\nu,\nu)
\bigl(\nu\eta_{T_\beta}\psi,\psi\bigr).
\end{align*}
We now evaluate the two coefficients in the adapted frame.  For the tangential coefficient, we have
\begin{align*}
\pi^{T_\beta}(\nu,e_A)
&=
k^{T_\beta}(\nu,e_A)
-
\bigl(\operatorname{tr}_{g}k^{T_\beta}\bigr)
g(\nu,e_A) \\
&=
k^{T_\beta}(\nu,e_A).
\end{align*}
For the normal coefficient, we have 
\begin{align*}
\pi^{T_\beta}(\nu,\nu)
&=
k^{T_\beta}(\nu,\nu)
-
\operatorname{tr}_{g}k^{T_\beta}\\
&=k^{T_\beta}(\nu,\nu)
-
\bigl(k^{T_\beta}(\nu,\nu)+\operatorname{tr}_{\Sigma}k^{T_\beta}\bigr)\\
&=-\operatorname{tr}_{\Sigma}k^{T_\beta}.
\end{align*}
Combining the ordinary and momentum parts gives
\begin{align}
\mathcal U_i\nu^i=
\bigl(D_\Sigma\psi,\psi\bigr)
-
\frac{1}{2}H|\psi|^2 
+
\frac{1}{2}
\sum_{\beta,A}
k^{T_\beta}(\nu,e_A)
\bigl(e_A\eta_{T_\beta}\psi,\psi\bigr)-
\frac{1}{2}
\sum_{\beta}
\operatorname{tr}_{\Sigma}k^{T_\beta}
\bigl(\nu\eta_{T_\beta}\psi,\psi\bigr).  \label{eq:boundary-current-before-projection}
\end{align}
If a self-adjoint endomorphism $C$ anticommutes with $\varepsilon_T$, then
  on the $+1$ bundle
  \[
    (C\psi,\psi)
      =(C\varepsilon_T\psi,\varepsilon_T\psi)
      =(\varepsilon_TC\varepsilon_T\psi,\psi)
      =-(C\psi,\psi)=0.
  \]
  Direct use of \eqref{eq:clifford-relations-prelim} shows
  \begin{align*}
   (e_A\eta_T)\varepsilon_T&=-\varepsilon_T(e_A\eta_T),\\
   (\nu\eta_{T_\beta})\varepsilon_T
      &=-\varepsilon_T(\nu\eta_{T_\beta}),\qquad \beta\geq2,\\
   (e_A\eta_{T_\beta})\varepsilon_T
      &=\varepsilon_T(e_A\eta_{T_\beta}),\qquad \beta\geq2.
  \end{align*}
 Thus the $T$-mixed terms and the $T^\perp$ trace terms vanish, ie.  
 \[
(e_A\eta_T\psi,\psi)=0, \quad (\nu\eta_{T_\beta}\psi,\psi)=0,
\qquad
\beta\geq2, 
\]
 whereas the $T^\perp$ mixed terms survive. For \(\beta=1\), since $(\nu\eta_T\psi,\psi)=|\psi|^2$ and 
 \[
(\mathcal A_T\psi,\psi)
=
\sum_{\beta=2}^{m}
\sum_{A=1}^{n-1}
(Q_T)_{\beta A}
\bigl(e_A\eta_{T_\beta}\psi,\psi\bigr).
\]
Therefore, we obtain
\[
B_{\Sigma,T}(\psi)
=
\mathcal U_i\nu^i
=
-\frac{1}{2}(H+\operatorname{tr}_{\Sigma}k_T)|\psi|^2
+
\frac{1}{2}(A_T\psi,\psi).
\]
This proves \eqref{eq:boundary-current-formula}.

By Lemma~2.1, applied in the adapted tangent and auxiliary
time frames, we have 
\[
(\mathcal A_T\psi,\psi)
\geq
-\lVert Q_T\rVert_{\mathrm{tr}}|\psi|^2.
\]
It follows that 
\begin{align*}
B_{\Sigma,T}(\psi)
&=
-\frac{1}{2}(H+\operatorname{tr}_{\Sigma}k_T)|\psi|^2
+
\frac{1}{2}(A_T\psi,\psi)\\
&\geq
-\frac{1}{2}(H+\operatorname{tr}_{\Sigma}k_T)|\psi|^2
-
\frac{1}{2}\lVert Q_T\rVert_{\mathrm{tr}}|\psi|^2 \\
&=
-\frac{1}{2}
\left(
H+\operatorname{tr}_{\Sigma}k_T+\lVert Q_T\rVert_{\mathrm{tr}}
\right)|\psi|^2.
\end{align*}
This is the desired estimate of  \eqref{eq:boundary-current-lower-bound}.
\end{proof}

\begin{corollary} \label{cor:boundary-positive}
  If
  \begin{equation}\label{eq:adapted-trapping-prelim}
    H+\operatorname{tr}_\Sigma k^T+\|\mathcal Q_T\|_{\mathrm{tr}}\leq0,
  \end{equation}
  then $\mathcal B_{\Sigma,T}(\psi)\geq0$ for every spinor satisfying
  $P_{T,-}\gamma\psi=0$.
\end{corollary}

\begin{corollary}[Compatible directional MOTS]\label{cor:directional-mots}
  If $\mathcal Q_T=0$, then
  \begin{equation}\label{eq:boundary-current-compatible}
    \mathcal B_{\Sigma,T}(\psi)=-\frac12\theta_T|\psi|^2.
  \end{equation}
  Hence a weakly $T$-outer trapped boundary, $\theta_T\leq0$, has the
  favorable sign, while a $T$-MOTS, $\theta_T=0$, makes the boundary current
  vanish.  A $T$-MOTS without $\mathcal Q_T=0$ does not in general control
  the surviving endomorphism $\mathcal A_T$.
\end{corollary}

\begin{remark}[Pointwise sharpness for the fixed projector]\label{rem:boundary-sharpness}
At a boundary point, set $S_+=\ker(\varepsilon_T-I)$. Then
\[
 \lambda_{\min}(\mathcal A_T|_{S_+})=-\|\mathcal Q_T\|_{\mathrm{tr}}.
\]
Indeed, a singular value decomposition gives
$\mathcal A_T=\sum_{s=1}^r\sigma_s\widehat e_s\widehat\eta_s$, where
$r=\operatorname{rank}\mathcal Q_T$, the $\widehat e_s$ are orthonormal
and tangential, and the $\widehat\eta_s$ correspond to orthonormal
directions in $T^\perp$. The operators $G_s=\widehat e_s\widehat\eta_s$
are commuting self-adjoint involutions preserving $S_+$. Tangential
multiplication by $\widehat e_s$ also preserves $S_+$, anticommutes with
$G_s$, and commutes with every $G_t$ for $t\ne s$. Moreover, $S_+$ is
nonzero, since the invertible operator $\nu$ interchanges the two
eigenspaces of $\varepsilon_T$. Starting with a nonzero joint eigenspace
in $S_+$ and applying these invertible sign changes produces a nonzero
joint eigenspace with all signs negative. This proves the asserted
minimum; when $r=0$, both sides vanish. Proposition~\ref{prop:boundary-current}
therefore gives the pointwise equivalence
\[
 \mathcal B_{\Sigma,T}(\psi)\geq0\ \text{for every }\psi\in S_+
 \quad\Longleftrightarrow\quad
 H+\operatorname{tr}_\Sigma k^T+\|\mathcal Q_T\|_{\mathrm{tr}}\leq0.
\]
This is an exact criterion for nonnegativity of the boundary current for
the fixed projector, not a necessary condition for the global mass
inequality.
\end{remark}

\begin{example}[Two-time data with a nonzero mixed boundary term]\label{ex:nonzero-mixed-boundary}
Let $n\geq3$, $m=2$, and $M=\{x\in\mathbb R^n:r=|x|\geq1\}$ with
its standard spin structure. Choose constants $b>0$ and
$a>1+nb/(n-2)$, and set
\[
 u(r)=1+ar^{2-n}-br^{1-n},\qquad g=u^{4/(n-2)}\delta.
\]
Since $a>b$, the function $u$ is positive and bounded above and below on
$M$. Thus $g$ is complete up to $\Sigma=\{r=1\}$ and asymptotically flat
of rate $n-2$. The conformal scalar-curvature and mean-curvature formulas
\cite[Exercises~1.34 and~2.14]{LeeGeometricRelativity2023} give
\begin{align*}
 R_g&=-\frac{4(n-1)}{n-2}u^{-(n+2)/(n-2)}\Delta_\delta u
     =\frac{4(n-1)^2b}{n-2}u^{-(n+2)/(n-2)}r^{-n-1}>0,\\
 H|_\Sigma&=(n-1)u(1)^{-2/(n-2)}
       \left(1+\frac{2u'(1)}{(n-2)u(1)}\right)\\
     &=(n-1)u(1)^{-n/(n-2)}
       \left(1-a+\frac{nb}{n-2}\right)<0,
\end{align*}
where $\nu=u(1)^{-2/(n-2)}\partial_r$ points into $M$.
Choose a smooth radial function $f$ equal to one near $r=1$ and zero for
$r\geq2$, and, for $\epsilon>0$, define
\[
 k^1=\epsilon f g,\qquad
 k^2=\epsilon f\bigl(dr\otimes dx^1+dx^1\otimes dr\bigr).
\]
The associated endomorphisms commute, since $(k^1)^\sharp=\epsilon fI$.
On the compact annulus $1\leq r\leq2$, let $c=\min R_g>0$.
There are constants $C_1,C_2$ independent of $\epsilon$ such that
\[
 \mu\geq\frac c2-C_1\epsilon^2,\qquad
 \|\mathcal J\|_{\mathrm{tr}}\leq C_2\epsilon.
\]
Thus the dominant energy condition holds for all sufficiently small
$\epsilon>0$; outside that annulus it follows from $\mu=R_g/2>0$
and $\mathcal J=0$. The tensors $k^\alpha$ are compactly supported,
$R_g=O(r^{-n-1})$, and hence all decay and integrability assumptions hold.
Take $T=T_1$. For $X\in T\Sigma$ one has
\[
 \mathcal Q_T(X)=\epsilon u(1)^{-2/(n-2)}dx^1(X)T_2,
 \qquad \operatorname{tr}_\Sigma k^T=(n-1)\epsilon.
\]
In particular, $\mathcal Q_T\ne0$ away from the two poles of $x^1|_\Sigma$.
Since $H$ is strictly negative and $\|\mathcal Q_T\|_{\mathrm{tr}}=O(\epsilon)$
uniformly on $\Sigma$, the adapted trapping condition also holds for
sufficiently small $\epsilon$. The tensors $k^1,k^2$ are linearly
independent wherever this mixed component is nonzero, so no fixed
rotation of the time directions reduces these data to one nonzero
tensor. This example verifies that the hypotheses allow a genuinely
nonzero mixed boundary term. Its ADM momenta vanish because the tensors
are compactly supported; it does not establish sharpness of the global
mass inequality.
\end{example}

\section{The Dirac--Witten boundary problem}\label{sec:boundary-problem}

Let $\mathcal H$ be the completion of smooth spinors compactly supported in
$\overline M$ in the energy norm
\begin{equation}\label{eq:energy-space}
  \|u\|_{\mathcal H}^2:=\int_M|\overline\nabla u|^2\,dV_g.
\end{equation}
The weighted Poincar\'e estimate makes this a norm and gives
\begin{equation}\label{eq:energy-weight-control}
  \int_M(1+r)^{-2}|u|^2\,dV_g
  \leq C\int_M|\overline\nabla u|^2\,dV_g.
\end{equation}
Consequently, on the end, $\mathcal H$ is norm-equivalent to
$W^{1,2}_{\delta_*}$ with $\delta_*=(2-n)/2$.  On compact sets its elements
are locally $H^1$, and their traces lie in $H^{1/2}(\Sigma)$.  Define
 \begin{align}\label{eq:energy-boundary-space}
  \mathcal H_T
  &:= \overline{\{u\in C_c^\infty(\overline M,\mathbb S):
             P_{T,-}\gamma u=0\}}^{\,\mathcal H}\notag \\
  &=\{u\in\mathcal H:P_{T,-}\gamma u=0
       \text{ in }P_{T,-}H^{1/2}(\Sigma)\}.
 \end{align}
For completeness, the second line really is the closure appearing in the
first.  Trace continuity gives the inclusion from left to right.  Conversely,
let $u$ lie in the trace kernel on the right.  First cut off $u$ on the
asymptotically flat end; the weighted Poincar\'e inequality and the usual
annular cutoff estimate give convergence in $\mathcal H$.  Approximate the
cutoff spinor by smooth compactly supported spinors in $H^1$ on a compact
manifold with collar.  Their $P_{T,-}$ traces converge to zero in
$H^{1/2}(\Sigma)$.  A bounded right inverse for the $H^1$ trace map, supported
in that collar, lifts and subtracts these small projected traces.  After a
final smoothing within the smooth subbundle $\operatorname{ran}P_{T,-}$, this
produces smooth compactly supported approximants satisfying
$P_{T,-}\gamma u_j=0$. 
We write $H_*^{1/2}(\Sigma)$ for their spectral graph trace space.  For the
present smooth compact boundary and self-adjoint tangential Dirac operator,
it equals $H^{1/2}(\Sigma)$ as a set, with an equivalent norm.
\begin{proposition}[Energy-class Dirac–Witten boundary isomorphism] \label{prop:weighted-solvability}
Assume that the multiple-time dominant energy condition \eqref{eq:dec-prelim},
pairwise commutativity of the \(k^\alpha\), and the adapted trapping
condition \eqref{eq:adapted-trapping-prelim}
\[
H+\operatorname{tr}_{\Sigma}k_T
+\lVert Q_T\rVert_{\mathrm{tr}}
\leq 0.
\]
holds. On each connected component of \(\Sigma\), let \(T\) be parallel in the
auxiliary flat time bundle, and let
\[
P_{T,-}
=
\frac{1}{2}\bigl(I-\varepsilon_T\bigr),
\qquad
\varepsilon_T=\nu\eta_T.
\]
Then
\[
 \overline D:\mathcal H_T\longrightarrow L^2(M,\overline S)
\]
is a bounded isomorphism. In particular, for every $f\in L^2(M,\overline S)$
there is a unique $u\in\mathcal H_T$ satisfying
\[
 \overline Du=f,\qquad P_{T,-}\gamma u=0,
 \qquad \|u\|_{\mathcal H}\leq\|f\|_{L^2}.
\]
The solution satisfies $\overline\nabla u\in L^2$ and
$(1+r)^{-1}u\in L^2$ which lies in $H^1_{\mathrm{loc}}$ up to $\Sigma$. 
If $f$ is smooth locally up to $\Sigma$, then so is $u$.
\end{proposition}
\begin{proof}
In an asymptotic orthonormal spin frame,
\[
 \overline\nabla_{\partial_i}u=\partial_i u+\Gamma_i u,
 \qquad |\Gamma_i|=O(r^{-1-\tau}).
\]
For a smooth compactly supported spinor, integration
by parts along a radial ray gives
\begin{align*}
 (n-2)\int_R^\infty r^{n-3}|u|^2\,dr
 &=-R^{n-2}|u(R)|^2
   -2\int_R^\infty r^{n-2}(\partial_ru,u)\,dr\\
 &\leq 2\int_R^\infty r^{n-2}|\partial_ru|\,|u|\,dr.
\end{align*}
Integrating over the sphere and applying Cauchy--Schwarz yields
\[
 \|r^{-1}u\|_{L^2(r>R,dx)}
 \leq\frac{2}{n-2}\|\partial_ru\|_{L^2(r>R,dx)}.
\]
The Euclidean and $g$ norms and measures are uniformly equivalent on the
end. Since $\partial_ru=\overline\nabla_{\partial_r}u-\Gamma_ru$, this
implies that for sufficiently large $R$, 
\begin{align*}
 \|r^{-1}u\|_{L^2(r>R)}
 &\leq C\|\overline\nabla u\|_{L^2(r>R)}
      +CR^{-\tau}\|r^{-1}u\|_{L^2(r>R)}\\
 &\leq C\|\overline\nabla u\|_{L^2(r>R)}.
\end{align*}
To control the compact part, we choose a connected compact region $K$ containing
that part and a fixed annulus $A\subset\{R<r<2R\}$. We claim that 
\[
 \|u\|_{L^2(K)}
 \leq C_K\bigl(\|\overline\nabla u\|_{L^2(K)}
                  +\|u\|_{L^2(A)}\bigr).
\]
Otherwise, after normalization there would be a sequence $u_j$ with
$\|u_j\|_{L^2(K)}=1$ and both terms on the right tending to zero.
Since $\overline\nabla-\nabla$ is bounded on $K$, this sequence is
bounded in $H^1(K)$. By Rellich compactness, there exists a subsequence converging
strongly in $L^2(K)$ to a spinor $u$ with
\[
 \|u\|_{L^2(K)}=1,\qquad
 \overline\nabla u=0,\qquad u|_A=0.
\]
Locally the equations $\partial_i u=-\Gamma_i u$ bootstrap its Sobolev regularity.
Uniqueness of parallel transport then implies $u=0$ on connected $K$,
a contradiction. Combining this compact estimate with the exterior
estimate proves \eqref{eq:energy-weight-control}.

On any compact region $K$, boundedness of
$\overline\nabla-\nabla$ and \eqref{eq:energy-weight-control} now gives
\[
 \|u\|_{H^1(K)}\leq C_K\|u\|_{\mathcal H}.
\]
Thus the completion embeds into local $H^1$, and the trace map into
$H^{1/2}(\Sigma)$ is continuous. By \eqref{eq:energy-boundary-space}, every
$u\in\mathcal H_T$ satisfies $P_{T,-}\gamma u=0$ in the trace sense.

Next, let $u\in C_c^\infty(\overline M,\overline S)$ satisfy the homogeneous
boundary condition. We integrate the Witten identity in Proposition 3.1, and after rearrangment, we have 
\begin{equation}\label{eq:coercive-identity}\begin{aligned}
 \|\overline Du\|_{L^2}^2
 &=\|u\|_{\mathcal H}^2
 +\frac12\int_M\left(
    \mu|u|^2+\sum_\alpha
    \langle J^\alpha,X^\alpha(u)\rangle\right)\,dV_g
 +\int_\Sigma\mathcal B_{\Sigma,T}(u)\,dA.
\end{aligned}\end{equation}
Note that by the trace-norm Clifford estimate and t
\[
 \mu|u|^2+\sum_\alpha\langle J^\alpha,X^\alpha(u)\rangle
 \geq(\mu-\|\mathcal J\|_{\mathrm{tr}})|u|^2\geq0.
\]
The boundary-current estimate from Section 4 gives
\[
 \mathcal B_{\Sigma,T}(u)
 \geq-\frac12
 \left(H+\operatorname{tr}_\Sigma k^T+\|Q_T\|_{\mathrm{tr}}\right)|u|^2
 \geq0.
\]
Consequently
\[
 \|u\|_{\mathcal H}\leq\|\overline Du\|_{L^2}.
\]
Conversely, the Clifford relations and Cauchy--Schwarz give
\[
 |\overline Du|^2
 =\left|\sum_i e_i\overline\nabla_i u\right|^2
 \leq n\sum_i|\overline\nabla_i u|^2.
\]
Hence $\overline D:\mathcal H\to L^2$ is bounded, and density in the
definition \eqref{eq:energy-boundary-space} extends the coercive estimate to every
$u\in\mathcal H_T$:
\[
 {\quad
 \|u\|_{\mathcal H}\leq\|\overline Du\|_{L^2}
 \leq\sqrt n\,\|u\|_{\mathcal H}.
 \quad}
\]
The first inequality proves injectivity and the range
is closed.\\
Suppose $w\in L^2(M,\overline S)$ is orthogonal to the range. Then
\[
 \int_M(w,\overline Dv)\,dV_g=0
 \qquad\text{for every }v\in\mathcal H_T.
\]
In particular, this holds for every smooth compactly supported $v$ with
$P_{T,-}\gamma v=0$. Interior test spinors first imply
$\overline Dw=0$ distributionally in the interior. At the boundary this
is the weak adjoint problem. Lemma 4.1 identifies its boundary condition
as the same homogeneous chirality condition.\\
We apply the weak-to-strong boundary regularity theorem
\cite[Theorem~6.4]{BartnikChrusciel2003Applications}, using the numbering
of the preprint with applications. In inward normal coordinates the
self-adjoint tangential operator is $-D_\Sigma$. Let $p$ and $q$ be its
positive and negative spectral projections, and let $z_\pm$ project onto
the $\varepsilon_T=\pm1$
subspaces of $\ker D_\Sigma$. Anticommutation in \eqref{eq:4.6} shows that
$\varepsilon_Tp=q\varepsilon_T$ and that $\varepsilon_T$ preserves
$\ker D_\Sigma$. With $P=p+z_-$ and $K=\varepsilon_Tq$, the chirality
condition for a boundary spinor $\varphi$ is equivalent to
\[
 P\varphi=K(I-P)\varphi.
\]
Both $K$ and its adjoint preserve the spectral $H^{1/2}$ spaces, since
$\varepsilon_T$ commutes with $|D_\Sigma|$. Thus the graph hypotheses
hold, including at zero eigenvalues. The smooth Dirac coefficients and
the smooth zero-order term $\overline D-D$ satisfy the local coefficient
hypotheses of that theorem.

The weak identity also holds for compactly supported $H^1$ test spinors
with the same boundary condition, as in the definition of weak solution
in that reference. This follows by local $H^1$ approximation: use smooth
bundle frames extending $S_+\oplus S_-$ from the boundary into a collar,
approximate the $S_+$ components by smooth functions, and the zero-trace
$S_-$ components by smooth zero-trace functions. A partition of unity
preserves the homogeneous condition. This local fact does not require
identifying $\mathcal H_T$ with the full global trace kernel.

The boundary theorem therefore gives
\[
 w\in H^1_{\mathrm{loc}}(\overline M,\overline S),
 \qquad \overline Dw=0,
 \qquad P_{T,-}\gamma w=0.
\]
Since the coefficients and the local elliptic boundary projector are
smooth, local elliptic boundary regularity bootstraps this homogeneous
solution to smoothness up to $\Sigma$.

Choose a smooth cutoff $\chi_R$ equal to one on the compact part and on
$\{r\leq R\}$, equal to zero on $\{r\geq2R\}$, and satisfying
$|d\chi_R|\leq C/R$. The smooth spinor $\chi_Rw$ is compactly supported
and satisfies $P_{T,-}\gamma(\chi_Rw)=0$. It therefore belongs directly
to the defining class in \eqref{eq:energy-boundary-space}. By the Leibniz rule, we have 
\[
 \overline D(\chi_Rw)
 =\chi_R\overline Dw+\sum_i e_i(e_i\chi_R)w
 =\sum_i e_i(e_i\chi_R)w.
\]
If $V_R=\sum_i(e_i\chi_R)e_i$, then the Clifford relations give
$V_R^*V_R=|d\chi_R|^2I$. Coercivity therefore implies
\begin{align*}
 \|\chi_Rw\|_{\mathcal H}^2
 &\leq\|\overline D(\chi_Rw)\|_{L^2}^2\\
 &=\int_M|d\chi_R|^2|w|^2\,dV_g\\
 &\leq\frac{C}{R^2}
       \int_{\{R<r<2R\}}|w|^2\,dV_g
 \longrightarrow0,
\end{align*}
where $w\in L^2$ is used in the last step. Applying \eqref{eq:energy-weight-control} to
$\chi_Rw$ gives
\[
 \|(1+r)^{-1}\chi_Rw\|_{L^2}\longrightarrow0.
\]
For every fixed compact region $K$, the cutoff equals one on $K$ when
$R$ is sufficiently large. Hence
\[
 \int_K(1+r)^{-2}|w|^2\,dV_g=0, 
\]
and $M$ by such regions proves $w=0$.\\
The range of $\overline D:\mathcal H_T\to L^2$ is closed and has zero
orthogonal complement, so it is all of $L^2$. Together with injectivity,
this proves the isomorphism. For its solution,
\[
 \|u\|_{\mathcal H}\leq\|\overline Du\|_{L^2}=\|f\|_{L^2},
 \qquad
 \|(1+r)^{-1}u\|_{L^2}\leq C\|f\|_{L^2}.
\]
Local $H^1$ regularity follows from the energy-space construction. Smooth
local $f$ gives smoothness up to the boundary by the same local elliptic
regularity for the homogeneous boundary projector. This proves all the
assertions in the proposition. 
\end{proof}

\begin{corollary}\label{cor:witten-spinor}
Assume the hypotheses of Proposition~5.1.  For every constant spinor $\psi^\infty$ on the asymptotic model there is a
  unique affine energy solution \(\psi\in\psi_0+\mathcal H_T\) such that 
  \begin{equation}\label{eq:witten-spinor-problem}
    \overline D\psi=0,\qquad
    P_{T,-}\gamma\psi=0,\qquad
    \psi=\psi_0+\chi, 
  \end{equation}
  where $\psi_0$ is any smooth extension that vanishes near $\Sigma$ and
  equals $\psi^\infty$ in the asymptotic spin frame outside a compact set. The resulting spinor is independent of the chosen extension \(\psi_0\).
Thus, for every \(\psi^\infty\), there is a unique solution in the
affine energy class \(\psi_0+\mathcal H_T\).  
\end{corollary}

\begin{proof}
    Choose a smooth cutoff extension \(\psi_0\). Outside a compact set,
\(\psi_0=\psi^\infty\) has constant components in the asymptotic spin
frame. Since $g_{ij}-\delta_{ij}=O_{2,\gamma}(r^{-\tau})$, the asymptotic spin connection coefficients are
\(O(r^{-\tau-1})\). Moreover, $k^\alpha=O(r^{-\tau-1})$. It follows from the equations \eqref{eq:modified-connection-prelim} and
  \eqref{eq:modified-dirac-zero-order} of \(\overline\nabla\) and
\(\overline D\) that
 \begin{equation}\label{eq:source-decay}
    f:=-\overline D\psi_0=O(r^{-\tau-1}).
  \end{equation}
Hence, we have 
\[
\begin{aligned}
\int_{\{r>R\}}|f|^2\,dV_g
&\le
C\int_R^\infty r^{n-1-2(\tau+1)}\,dr\\
&=
C\int_R^\infty r^{n-3-2\tau}\,dr
<\infty,
\end{aligned}
\]
because \(\tau>(n-2)/2\). The contribution on the compact part of
\(M\) is finite, so \(f\in L^2(M,S)\).

Since \(\psi_0\) vanishes in a boundary collar, $P_{T,-}\gamma\psi_0=0$. Proposition~~\ref{prop:weighted-solvability}, applied with homogeneous boundary data, gives a unique
\(\chi\in\mathcal H_T\) satisfying $\overline D\chi=f$. Then   $\psi=\psi_0+\chi$ satisfies \eqref{eq:witten-spinor-problem}, ie. 
\[
\overline D(\psi_0+\chi)
=
\overline D\psi_0+f
=
0, \quad P_{T,-}\gamma(\psi_0+\chi)=0.
\]
Finally, notice that a nonzero constant spinor does not belong to
\(L^2_{\delta_*}\), since
\[
\int_R^\infty
r^{-2}r^{n-1}\,dr
=
\int_R^\infty r^{n-3}\,dr
=
\infty.
\]
Thus the solution is correctly formulated in the affine class
\(\psi_0+\mathcal H_T\). 
If $\widetilde\psi_0$ is another such extension, then
$\psi_0-\widetilde\psi_0$ is smooth, compactly supported, and zero near
$\Sigma$, hence belongs to $\mathcal H_T$. The difference of the two
resulting solutions therefore lies in $\mathcal H_T$ and is annihilated
by $\overline D$. Injectivity in Proposition~\ref{prop:weighted-solvability}
shows that the final spinor $\psi$ is independent of the extension.

\end{proof}

\section{The mass identity and proof of Theorem~\ref{thm:main-boundary}}
\label{sec:positive-mass}

\begin{lemma}
\label{lem:energy-asymptotics}
  Assume the hypotheses of Proposition 5.1 and let $\psi=\psi_0+\chi$ be the Witten spinor of
  Corollary~\ref{cor:witten-spinor} with $\chi\in \mathcal H_T$.  Then
  \begin{equation}\label{eq:actual-decay}
    \overline\nabla\chi\in L^2(M),\qquad
    (1+r)^{-1}\chi\in L^2(M),
  \end{equation}
  and $\chi$ is smooth locally up to the boundary $\Sigma$.  The following
  energy-class mass identity holds:
  \begin{align}
   \frac{(n-1)\omega_{n-1}}{2}
      \left[
       E N(\psi^\infty)
       +\sum_{\alpha=1}^m
         \langle P^\alpha,X^\alpha(\psi^\infty)\rangle
      \right]\notag&=\int_M\left[
      |\overline\nabla\psi|^2
      +\frac12\mu|\psi|^2
      +\frac12\sum_\alpha
        \langle J^\alpha,X^\alpha(\psi)\rangle
     \right]dV_g\\
     & +\int_\Sigma\mathcal B_{\Sigma,T}(\psi)\,dA.
   \label{eq:energy-class-mass-identity}
  \end{align}
\end{lemma}

\begin{remark}
 \eqref{eq:energy-class-mass-identity} is a weak boundary
  functional statement.  No assertion that $\chi=O(r^{-q})$, or that the
  literal sphere flux of $\chi$ has a limit, is needed or implied.   
\end{remark}

This is the compact-boundary, energy-class analogue of the integral formula
in \cite[Theorem~3.6]{HirschPayneZhang2026}.  The ADM integrand is the same,
but the proof below adds the compact-boundary current and replaces HPZ's
claimed $C^{2,\alpha}_{-q}$ remainder by the decay actually supplied by the
boundary energy theory.

\begin{proof}
By Corollary~5.2, $\chi\in\mathcal{H}_T\subset\mathcal{H}$, and the energy norm is
\[
\lVert\chi\rVert_{\mathcal{H}}^2
=
\int_M
\bigl|\overline{\nabla}\chi\bigr|^2\,dV_g.
\]
Therefore, $\overline{\nabla}\chi\in L^2(M)$. By the weighted Poincar\'e inequality, 
\[
\int_M
(1+r)^{-2}|\chi|^2\,dV_g
\leq
C\int_M
\bigl|\overline{\nabla}\chi\bigr|^2\,dV_g.
\]
Hence, we also have $(1+r)^{-1}\chi\in L^2(M)$. 

  We now prove \eqref{eq:energy-class-mass-identity} by a density argument.   By
  \eqref{eq:energy-boundary-space}, choose
  $\chi_a\in C_c^\infty(\overline M,\mathbb S)$ satisfying
  \begin{equation}\label{eq:energy-approximation}
    P_{T,-}\gamma\chi_a=0, \qquad \|\chi_a-\chi\|_{\mathcal H}\longrightarrow0.
  \end{equation}
  Put $\psi_a=\psi_0+\chi_a$.  For each $a$, choose $R$ beyond the support
  of $\chi_a$. Let \(M_R\) denote the region bounded by the large coordinate sphere \(S_R\) with compact inner boundary \(\Sigma\), so that \(\partial M_R=S_R\sqcup\Sigma\). By the divergence theorem and Proposition~\ref{prop:witten-identity}, integrating over the
  region $M_R$ bounded by $S_R$ and $\Sigma$ gives
  \begin{align}
   \int_{S_R}\mathcal U_i(\psi_a)\nu_R^i\,dA
&=\int_{M_R}\left[
      |\overline\nabla\psi_a|^2-|\overline D\psi_a|^2
      +\frac12(\psi_a,\mathbf Q\psi_a)
     \right]dV_g +\int_\Sigma\mathcal B_{\Sigma,T}(\psi_a)\,dA,   
   \label{eq:approximate-integrated-identity}
  \end{align}
  where
    $\mathbf Q:=\mu I+\sum_{\alpha,i}J_i^\alpha e_i\eta_\alpha$

  We first evaluate the left side for $R\to\infty$.  Outside a compact set,
  $\psi_a=\psi_0=\psi^\infty$ in the asymptotic spin frame.  Write
  $h_{ij}=g_{ij}-\delta_{ij}$ and choose the symmetric orthonormal gauge
  \[
    E_i=\left(\delta_i^{\,j}-\frac12h_i^{\,j}\right)\partial_j
         +O(r^{-2\tau}).
  \]
  Its spin connection satisfies
  \[
    \nabla_i\psi^\infty
      =\frac14\omega_{ijk}E_jE_k\psi^\infty,\qquad \text{where }
    \omega_{ijk}
      =\frac12(\partial_jh_{ik}-\partial_kh_{ij})
       +O(r^{-2\tau-1}).
  \]
 
  For $
D=E_\ell\nabla_\ell$ and substituting into the first term of 
  \eqref{eq:corrected-witten-current}, followed by skew-symmetry in $j,k$,
  gives
\[
(E_iD\psi^\infty+\nabla_i\psi^\infty,\psi^\infty)
=
\left(\frac{1}{4}
\omega_{\ell jk}
E_iE_\ell E_jE_k\psi^\infty
+
\frac{1}{4}
\omega_{ijk}
E_jE_k\psi^\infty,\psi^\infty \right).
\]
Since \(\omega_{ijk}\) is skew-symmetric in \(j,k\), $\left(
\frac{1}{4}
\omega_{ijk}E_jE_k\psi^\infty,
\psi^\infty
\right)
=
0$. To handle the fourfold Clifford product, write $\omega^{(1)}$ for
the displayed first-order part of $\omega$. Symmetry of $h$ gives
\[
 \omega^{(1)}_{\ell jk}+\omega^{(1)}_{jk\ell}
 +\omega^{(1)}_{k\ell j}=0.
\]
Hence its completely antisymmetric part vanishes, and Clifford contraction
yields
\[
 D\psi^\infty
 =\frac14(\partial_i h_{jj}-\partial_j h_{ij})E_i\psi^\infty
 +O(r^{-2\tau-1}).
\]
Using $(E_iE_j\psi^\infty,\psi^\infty)=-\delta_{ij}|\psi^\infty|^2$
for the real pairing, we obtain
  \begin{align}
   (E_iD\psi^\infty+\nabla_i\psi^\infty,\psi^\infty)\nu_R^i&=
\frac{1}{2}
\omega_{\ell\ell i}\nu_R^i|\psi^\infty|^2
+
O(r^{-2\tau-1}) \notag\\
    &=\frac14(\partial_jg_{ij}-\partial_ig_{jj})\nu_R^i
       |\psi^\infty|^2+O(r^{-2\tau-1}).
   \label{eq:constant-energy-flux}
  \end{align}
  The momentum part is
  \begin{align}
   \frac12\sum_\alpha\pi^\alpha_{ij}\nu_R^i
      (E_j\eta_\alpha\psi^\infty,\psi^\infty)
    &=\frac12\sum_\alpha\pi^\alpha_{ij}\nu_R^i
      (e_j\eta_\alpha\psi^\infty,\psi^\infty)+O(r^{-2\tau-1}).
   \label{eq:constant-momentum-flux}
  \end{align}
  Since the area of $S_R$ is $O(R^{n-1})$ and
  $2\tau>n-2$, the integrated errors in
  \eqref{eq:constant-energy-flux} and
  \eqref{eq:constant-momentum-flux} tend to zero.  Using
  \eqref{eq:adm-energy-prelim}--\eqref{eq:adm-momentum-prelim} and symmetry
  of $\pi^\alpha$, we obtain
  \begin{align}
   \lim_{R\to\infty}\int_{S_R}\mathcal U_i(\psi_a)\nu_R^i\,dA
  =\frac{(n-1)\omega_{n-1}}2
      \left[
       E N(\psi^\infty)
       +\sum_\alpha\langle P^\alpha,
          X^\alpha(\psi^\infty)\rangle
      \right],
   \label{eq:constant-spinor-adm-functional}
  \end{align}
  where right side is independent of $a$.

  We next pass $a\to\infty$ on the right side of
  \eqref{eq:approximate-integrated-identity}.  From
  \eqref{eq:energy-approximation},
  $\overline\nabla\psi_a\to\overline\nabla\psi$ in $L^2$ and so $\int_M
\bigl|\overline{\nabla}\psi_a\bigr|^2\,
\longrightarrow
\int_M
\bigl|\overline {\nabla}\psi\bigr|^2\,$.   Also
  \[
    |\overline D(\psi_a-\psi)|
      =|e_i\overline\nabla_i(\psi_a-\psi)|
      \leq\sqrt n\,|\overline\nabla(\psi_a-\psi)|,
  \]
  so $\overline D\psi_a\to\overline D\psi=0$ in $L^2$ and thus $\int_M
\bigl|\widetilde{D}\psi_a\bigr|^2\,
\longrightarrow0$. 

  Lemma~\ref{lem:trace-norm-clifford} says that $\mathbf Q$ is a
  nonnegative self-adjoint endomorphism.  The DEC also gives
  $\|\mathbf Q\|_{\mathrm{op}}\leq2\mu$, so
  $\int(\psi_0,\mathbf Q\psi_0)\leq
2\int_M
\mu|\psi_0|^2<\infty$ as $\mu\in L^1(M)$ and
  $\psi_0$ is bounded.  For a compactly supported homogeneous-boundary
  spinor $v$ satisfying $P_{T,-}\gamma v=0$, \eqref{eq:coercive-identity} and
  $|\overline Dv|^2\leq n|\overline\nabla v|^2$ yield
\begin{align}\label{eq:Q-energy-control}
\int_M
(v,\textbf{Q}v)\,dV_g
\leq
2\int_M
\bigl|\overline{D}v\bigr|^2\,dV_g&\leq
2n\int_M
\bigl|\overline{\nabla}v\bigr|^2\,dV_g \notag  \\
&=
2n\lVert v\rVert_{\mathcal{H}}^2.
\end{align}
 The estimate makes $\mathbf Q^{1/2}$ a bounded map from the dense
  smooth homogeneous-boundary subspace of $\mathcal H_T$ to $L^2$.
  It therefore extends continuously to $\mathcal H_T$. Local $L^2$
  convergence and local boundedness of $\mathbf Q^{1/2}$ identify this
  extension with pointwise multiplication by $\mathbf Q^{1/2}$.
  Thus the estimate extends to every $v\in\mathcal H_T$.  The pointwise
  Cauchy--Schwarz inequality for the positive self-adjoint form $\mathbf Q$ gives
  \[
    |(u,\mathbf Qv)|
      \leq(u,\mathbf Qu)^{1/2}(v,\mathbf Qv)^{1/2}.
  \]
  After integration and an ordinary Cauchy--Schwarz inequality, this gives 
  \[
\left|
\int_M
(u,\textbf{Q}v)\,dV_g
\right|
\leq
\left(
\int_M
(u,\textbf{Q}u)\,dV_g
\right)^{1/2}
\left(
\int_M
(v,\textbf{Q}v)\,dV_g
\right)^{1/2}.
\]
Now let $v_a=\psi_a-\psi$, and we know from \eqref{eq:Q-energy-control} that 
\[
\int_M
(v_a,\textbf{Q}v_a)\,dV_g
\leq
2n\lVert v_a\rVert_{\mathcal{H}}^2
\longrightarrow0, \quad \int_M
(\psi,\textbf{Q}\psi)\,dV_g
<\infty
\]
Hence, 
\begin{align*}
&\left|
\int_M
\left[
(\psi_a,\textbf{Q}\psi_a)-(\psi,\textbf{Q}\psi)
\right]\,dV_g
\right| 
\leq
2
\left(
\int_M
(\psi,\textbf{Q}\psi)\,dV_g
\right)^{1/2}
\left(
\int_M
(v_a,\textbf{Q}v_a)\,dV_g
\right)^{1/2}
+
\int_M
(v_a,\textbf{Q}v_a)\,dV_g,
\end{align*}
which goes to $0$. Therefore,
\[
\int_M
(\psi_a,Q\psi_a)\,dV_g
\longrightarrow
\int_M
(\psi,Q\psi)\,dV_g.
\]
Finally, on a fixed compact collar \(K\) of \(\Sigma\), by the weighted Poincar\'e
estimate $\lVert v_a\rVert_{L^2(K)}
\leq
C_K\lVert v_a\rVert_{\mathcal{H}}$. In addition, 
$\nabla_iv_a
=
\overline{\nabla}_iv_a
-
\frac{1}{2}
k_{ij}^\alpha e_j\eta_\alpha v_a$, so we have 
\begin{align*}
\lVert\nabla v_a\rVert_{L^2(K)}
&\leq
\bigl\lVert\overline{\nabla}v_a\bigr\rVert_{L^2(K)}
+
C_K\lVert v_a\rVert_{L^2(K)} \\
&\leq
C_K\lVert v_a\rVert_{\mathcal{H}}.
\end{align*}
Thus $v_a\longrightarrow0$ in $H^1(K)$ implies that  $\gamma v_a\longrightarrow0$ in  \(L^2(\Sigma)\). It follows for the boundary current
\[
B_{\Sigma,T}(u)
=
-\frac{1}{2}(H+\operatorname{tr}_\Sigma k^T)|u|^2
+
\frac{1}{2}(A_Tu,u).
\]
we obtain
\begin{align*}
&\left|
\int_\Sigma
B_{\Sigma,T}(\psi_a)\,dA
-
\int_\Sigma
B_{\Sigma,T}(\psi)\,dA
\right|
\leq
C
\left(
\lVert\gamma\psi_a\rVert_{L^2}
+
\lVert\gamma\psi\rVert_{L^2}
\right)
\lVert\gamma\psi_a-\gamma\psi\rVert_{L^2},
\end{align*}
which tends to zero and so 
\[
\int_\Sigma
B_{\Sigma,T}(\psi_a)\,dA
\longrightarrow
\int_\Sigma
B_{\Sigma,T}(\psi)\,dA.
\]
The desired identity follows from passing first $R\to\infty$ and then $a\to\infty$ in
  \eqref{eq:approximate-integrated-identity}, and using
  $\overline D\psi=0$, gives \eqref{eq:energy-class-mass-identity}.

\end{proof}

\begin{lemma}[Choice of the asymptotic spinor]
\label{lem:asymptotic-spinor}
Let $\mathcal{P}
:=
\sum_{\alpha=1}^{m}
\sum_{i=1}^{n}
P_i^\alpha e_i\eta_\alpha$ 
act on the asymptotic enlarged spinor module $S$. Then $\lambda_{\min}(\mathcal{P})
=
-\lVert P\rVert_{\mathrm{tr}}$ and so there is a nonzero constant spinor $\psi^\infty$ such that
  \begin{equation}\label{eq:momentum-saturation}
    \sum_{\alpha=1}^m
      \langle P^\alpha,X^\alpha(\psi^\infty)\rangle
      =-\|\mathcal P\|_{\mathrm{tr}}N(\psi^\infty).
  \end{equation}
\end{lemma}
\begin{proof}
  Take an SVD of $\mathcal P$ as  in
  Lemma~\ref{lem:trace-norm-clifford}, $P=U^{T}\Lambda V$ with $U\in O(m), V\in O(n)$, and $\Lambda_{ss}=\sigma_s(P)$ for $1\leq s\leq m$. Consider the rotated Clifford generators $\widetilde{\eta}_s$ and $\widetilde{e}_s$. Then for $G_s:=\widehat e_s\widehat\eta_s$,  we may write 
\begin{align*}
\mathcal{P}
=
\sum_{\alpha,i}
P_{\alpha i}e_i\eta_\alpha &=
\sum_{\alpha,i}
\left(
\sum_{s=1}^{m}
U_{s\alpha}\sigma_s(P)V_{si}
\right)
e_i\eta_\alpha \\
&=
\sum_{s=1}^{m}
\sigma_s(P)
\sum_{\alpha,i}
V_{si}U_{s\alpha}e_i\eta_\alpha \\
&=
\sum_{s=1}^{m}
\sigma_s(P)\hat{e}_s\hat{\eta}_s\\
&=\sum_{s=1}^{m}
\sigma_s(P)G_s.
\end{align*}
  Every $G_s$ is a self-adjoint involution with  eigenvalues
\(+1\) and \(-1\) by the calculation in
  Lemma~\ref{lem:trace-norm-clifford}, ie. $G_s^*=G_s$ and $G_s^2=I$.   If $s\neq t$, then we compute that 
  \[
   G_sG_t=-\widehat e_s\widehat e_t
             \widehat\eta_s\widehat\eta_t
          =G_tG_s,
  \]
For $\varepsilon
=
(\varepsilon_1,\ldots,\varepsilon_m)
\in\{\pm1\}^m$, let 
\[
E_\varepsilon
:=
\bigcap_{s=1}^{m}
\ker\bigl(G_s-\varepsilon_sI\bigr).
\]
Then $\psi\in E_\varepsilon$ implies that $G_s\psi=\varepsilon_s\psi$ for every $s$. Simultaneous diagonalization gives the orthogonal decomposition
\[
S_\infty
=
\bigoplus_{\varepsilon\in\{\pm1\}^m}
E_\varepsilon,
\]
where some \(E_\varepsilon\) could initially be zero. Since the entire spinor space is nonzero, at least one joint eigenspace is
nonzero. Note that $G_s\hat{e}_s
=
-\hat{e}_sG_s$ and $G_t\hat{e}_s
=
\hat{e}_sG_t$ for \(t\neq s\). Since $\hat{e}_s$ is invertible, it follows that 
\[
\hat{e}_s
:
E_{(\varepsilon_1,\ldots,\varepsilon_s,\ldots,\varepsilon_m)}
\longrightarrow
E_{(\varepsilon_1,\ldots,-\varepsilon_s,\ldots,\varepsilon_m)},
\]
is an isomorphism. Then every joint eigenspace $E_\epsilon$ is nonzero and have equal positive dimension. Now we choose the asymptotic spinor
$0\neq\psi_\infty
\in
E_{(-1,\ldots,-1)}$. We compute that 
\begin{align*}
\sum_{\alpha=1}^{m}
\left\langle
P^\alpha,X^\alpha(\psi_\infty)
\right\rangle
&=
\sum_{\alpha=1}^{m}
\sum_{i=1}^{n}
P_{\alpha i}X_i^\alpha(\psi_\infty) \\
&=
\left(
\sum_{\alpha,i}
P_{\alpha i}e_i\eta_\alpha\psi_\infty,
\psi_\infty
\right)\\
&=
\sum_{s=1}^{m}
\sigma_s(P)
\bigl(
G_s\psi_\infty,\psi_\infty
\bigr) \\
&=
\sum_{s=1}^{m}
\sigma_s(P)
\bigl(
-\psi_\infty,\psi_\infty
\bigr) \\
&=
-\sum_{s=1}^{m}
\sigma_s(P)|\psi_\infty|^2 \\
&=
-\lVert P\rVert_{\mathrm{tr}}
N(\psi_\infty).
\end{align*}
This proves the lemma.
\end{proof}

\begin{proof}[Proof of Theorem~\ref{thm:main-boundary}]
  Choose a nonzero constant spinor $\psi^\infty\neq0$ as in 
  Lemma~\ref{lem:asymptotic-spinor} and let \(\psi=\psi_0+\chi\) be the corresponding Witten spinor
constructed in Corollary~\ref{cor:witten-spinor}. 
By Lemma
  \ref{lem:energy-asymptotics}, we have 
  \begin{align*}
\frac{(n-1)\omega_{n-1}}2
\left[
E N(\psi^\infty)
+
\sum_{\alpha=1}^{m}
\left\langle
P^\alpha,X^\alpha(\psi^\infty)
\right\rangle
\right]
&=
\int_M
\left[
|\overline\nabla\psi|^2
+
\frac12
\left(
\mu|\psi|^2
+
\sum_{\alpha=1}^{m}
\left\langle
J^\alpha,X^\alpha(\psi)
\right\rangle
\right)
\right]dV_g\\
&\quad+
\int_\Sigma
\mathcal B_{\Sigma,T}(\psi)\,dA.
\end{align*}.  
The first term is nonnegative pointwise by Lemma~\ref{lem:trace-norm-clifford} and the dominant energy condition. The 
second expression is nonnegative by
  Corollary~\ref{cor:boundary-positive} and the adapted trapping condition. Therefore
  \[
    \frac{(n-1)\omega_{n-1}}2
      \left[
       E N(\psi^\infty)
       +\sum_\alpha
         \langle P^\alpha,X^\alpha(\psi^\infty)\rangle
      \right]\geq0.
  \]
  Substituting of \eqref{eq:momentum-saturation} gives
\begin{align*}
\frac{(n-1)\omega_{n-1}}2
\left[
E N(\psi^\infty)
+
\sum_{\alpha=1}^{m}
\left\langle
P^\alpha,X^\alpha(\psi^\infty)
\right\rangle
\right]&=\frac{(n-1)\omega_{n-1}}2
\left(
E-\|P\|_{\mathrm{tr}}
\right)
N(\psi^\infty)\\
&\ge0.
\end{align*}
  Since $N(\psi^\infty)>0$, it follows that $E\geq\|\mathcal P\|_{\mathrm{tr}}$ as desired. 
\end{proof}
\newpage

\bibliographystyle{plain}
\bibliography{bibliography}

\end{document}